\documentclass{article}
\usepackage[utf8]{inputenc}
\usepackage[margin = 1in]{geometry}
\usepackage{amsfonts, amsmath, amssymb, amsthm}
\usepackage{enumitem}
\usepackage[authoryear]{natbib}
\usepackage{hyperref}
\usepackage{xcolor}
\usepackage{graphicx}
\usepackage{caption}
\usepackage{subcaption}
\usepackage{appendix}
\usepackage{titlesec}
\usepackage{setspace}
\usepackage{authblk}

\theoremstyle{plain}
\newtheorem{theorem}{Theorem}[section]
\newtheorem{lemma}[theorem]{Lemma}

\theoremstyle{definition}
\newtheorem{assumption}{Assumption}[subsection]
\newtheorem{remark}{Remark}[subsection]

\newcommand{\tr}[1]{\operatorname{tr} \left\{ #1 \right\}}

\allowdisplaybreaks
\title{Frequentist Oracle Properties of Bayesian Stacking Estimators}
\author[1]{Valentin Zulj\thanks{ \textbf{Corresponding author. E-mail:} \href{mailto:valentin.zulj@statistik.uu.se}{valentin.zulj@statistik.uu.se}}}
\author[1, 2]{Shaobo Jin}
\author[1]{Måns Magnusson}

\affil[1]{Department of Statistics, Uppsala University}
\affil[2]{Department of Mathematics, Uppsala University}

\date{}

\begin{document}
	\maketitle 
	
	\begin{abstract}
		\noindent Compromise estimation entails using a weighted average of outputs from several candidate models, and is a viable alternative to model selection when the choice of model is not obvious. As such, it is a tool used by both frequentists and Bayesians, and in both cases, the literature is vast and includes studies of performance in simulations and applied examples. However, frequentist researchers often prove oracle properties, showing that a proposed average asymptotically performs at least as well as any other average comprising the same candidates. On the Bayesian side, such oracle properties are yet to be established. This paper considers Bayesian stacking estimators, and evaluates their performance using frequentist asymptotics. Oracle properties are derived for estimators stacking Bayesian linear and logistic regression models, and combined with Monte Carlo experiments that show Bayesian stacking may outperform the best candidate model included in the stack. Thus, the result is not only a frequentist motivation of a fundamentally Bayesian procedure, but also an extended range of methods available to frequentist practitioners.  \\
		
	\end{abstract}
	
	\section{Introduction}
	
	The issue of specifying a model is common to many applications of statistical methods \citep[see e.g.][]{Hansen2005, deLuna2011, Ren2023}. In many cases, there is no obvious or straightforward choice, and thus, much effort has been put into providing off-the-shelf procedures that help specify an apt model \citep{Rao2001}. However, the suitability of such schemes, and of model selection as a general practice, has been questioned in the last decades. The criticisms range from \cite{Breiman1996}, who illustrates the instability of certain types of model selection in prediction problems, to \cite{Draper1995}, who reasons that the very fact that model selection is required \emph{implies structural uncertainty that in general is not fully assessed and propagated}, making for flawed inferences. The problems raised are serious, and have acted as a hotbed for the development of \textit{compromise estimation}, an alternative modeling strategy that considers a trade-off between several candidate models. This type of approach has been studied in both Bayesian \citep{HoetingTutorial, LeClarke2017} and frequentist \citep{Wolpert1992, HjortClaeskens2003} settings.
	
	Compromise estimators, also known as model averages or stacking estimators, have proved to be particularly fruitful as tools for improving predictive performance. In the frequentist literature, several types of model averages have been considered, often resulting in proofs of a certain type of asymptotic optimality known as an \textit{oracle property} \cite[see e.g.][]{HansenRacine2012, Gao2019}. Among other things, this means that, as the sample size $n \rightarrow \infty$, no single candidate model individually outperforms the compromise. Prominent examples include \cite{Hansen2007}, who develops Mallows' model averaging, \cite{HansenRacine2012}, who introduce and prove certain oracle properties of jackknife model averaging (JMA), and \cite{ZhangZouCarroll2015}, who develop model averaging based on a Kullback-Leibler criterion. Moreover, \cite{Zhao2020} consider jackknife averaging of ridge regression models, \cite{ZhaoZhouYang2019} apply JMA to discrete choice models, \cite{AndoLi2017} develop a weighting scheme for high-dimensional generalized linear models, while \cite{HjortClaeskens2003} and \cite{CharkhiClaeskensHansen2016} study the properties of model averages in a $\sqrt{n}$ local asymptotic framework. 
	
	From a Bayesian perspective, model averages \citep[see e.g.][]{RafteryMadiganHoeting1997, HoetingTutorial} traditionally use posterior model probabilities to weight candidates together, resulting in a posterior distribution that compromises between candidate models. This procedure is known as Bayesian model averaging (BMA). When the set of candidate models includes a model that correctly specifies the data generating distribution, \cite{ClydeIversen2013} state that BMA is the optimal action, under the quadratic loss function, for prediction of new outcomes. In such cases, the weights used to combine models are well defined, and may be computed either analytically or by use of simulation. However, when the set of candidates is not believed to contain the data generating distribution, \cite{ClydeIversen2013} argue that the specification of prior model probabilities is no longer reasonable, since such probabilities reflect the belief that each candidate represents the data generating mechanism. In practice, specifying the correct model is rarely considered realistic. Thus, in order to retain the benefits of compromise estimation in this context, a strategy for \textit{estimating} suitable weights is required. This motivates a Bayesian  re-interpretation of stacking, a frequentist approach to compromise estimation, first introduced by \cite{Wolpert1992}.
	
	Stacking has been shown to work well when used to combine misspecified candidate models, leading  \cite{ClydeIversen2013} to propose using Bayesian stacking to average posterior means,  and \cite{Yao2018} to suggest using stacking to combine posterior predictive distributions. As opposed to BMA, Bayesian stacking uses weights estimated by cross-validation, and thus does not depend on the use of posterior model probabilities, which are often difficult to correctly specify, and to obtain \citep{Yao2018}. The Bayesian interpretation of stacking is theoretically supported by \cite{LeClarke2017}, who provide Bayesian grounds for using cross-validation to estimate  weights, by showing that, under certain conditions, leave-one-out risk estimation is consistent, and that the stacking weights estimated using leave-one-out cross-validation asymptotically minimize a posterior expected loss. The asymptotic results are derived under the posterior marginal distribution of the outcome $\boldsymbol{y}$, rather than the data generating distribution of $\boldsymbol{y}$. This is not coherent with frequentist asymptotic theory, making it pertinent to asymptotically evaluate Bayesian stacking under the generating probability of $\boldsymbol{y}$. In particular, this may provide frequentist support for the use of Bayesian stacking, and also offer new tools to frequentist practitioners.
	
	Bayesian stacking of posterior means is similar to the frequentist procedure of jackknife model averaging, in that they employ similar cross-validation strategies to estimate weights. The literature concerning JMA is vast, and includes proofs of theorems asserting properties of asymptotic optimality, also known as \textit{oracle properties} \citep[see e.g.][]{Hansen2007, HansenRacine2012}. Due to the similarity of Bayesian stacking to JMA, certain asymptotic properties dervied for JMA can be generalized to cover a range of Bayesian stacking approaches. As such, this paper aims to 
	\begin{enumerate}
		
		\item Prove an oracle property that applies to Bayesian stacking of a general form of Bayesian linear regression models, assuming $\mathcal{N}(\boldsymbol{0}, \boldsymbol{S})$ priors on the regression coefficients, and 
		
		\item Discuss the application of existing oracle properties to Bayesian stacking of logistic regression models, as well as linear regression models assuming a more general class of priors.
		
	\end{enumerate}
	The first contribution builds on Theorem 2 of \cite{Zhao2020}, and extends this result to encompass Bayesian stacking of linear regression models specified as above. Here, the symmetric and positive definite covariance matrix $\boldsymbol S$ can be specified such that $\boldsymbol \beta$ follows the \cite{Zellner1986} $g$ prior, for example. For the second contribution, the oracle property derived by \cite{ZhaoZhouYang2019} is discussed in the context of Bayesian logistic regression. The discussion is centered around posterior concentration results given by \cite{Ramamoorthi2015} and \cite{KleijnvanderVaart2012}, which also provides some grounds for using $\mathcal{T}$ priors, commonly recommended in Bayesian settings (see e.g. Stan's \href{https://github.com/stan-dev/stan/wiki/Prior-Choice-Recommendations}{Prior Choice Recommendations}), to formulate candidate models. To complement the theoretical study with empirical results, the paper also provides a set of simulated examples, where the proposed stacking estimators are compared to the best candidate in each stack.
	
	The rest of the paper is organized as follows. Section \ref{sec:prelim} provides a detailed introduction to compromise estimators in general, as well as more detailed discussion of Jackknife averaging and Bayesian stacking in particular. Section \ref{sec:optimality} introduces the concept of oracle properties, and provides theorems establishing the results discussed in the preceding paragraph. Section \ref{sec:simulations} gives details of the Monte Carlo experiments performed to investigate the finite sample performance of the methods studied, and, lastly, Section \ref{sec:conclusions} summarizes the findings of the paper.

		\section{Notation and Preliminaries} \label{sec:prelim}
	
	Suppose that, for each unit $i = 1, \ldots, n$, a random outcome, $y_i$, and a $p$-dimensional vector of fixed covariates, $\boldsymbol{x}_i$, are observed, and define $\boldsymbol{y} = (y_1, \ldots, y_n)^T$ and $\boldsymbol{X} = (\boldsymbol{x}_1, \ldots, \boldsymbol{x}_n)^T$. Moreover, let $\mathcal{M} = \{M_1, \ldots, M_K\}$ be a collection of candidate models considered suitable for predicting new observations of the outcome. For example, the $M_k$ may differ in the sense that they use different sets of covariates or different values of some tuning or regularization parameter.
	
	In accordance with \cite{LeClarke2017}, the set of candidates $\mathcal{M}$ is classified as either \textit{closed}, \textit{complete}, or \textit{open}. The classes differ in the sense that
	
	\begin{itemize}
		\item $\mathcal{M}$-closed represents the situation where one $M_k \in \mathcal{M}$ is believed to cover the data generating mechanism, although there is no knowledge of which candidate it is, whereas 
		
		\item $\mathcal{M}$-complete represents the case where data generating mechanism can be conceptualized, but cannot be included in $\mathcal{M}$ to due to, for example, complexity or practical feasibility. Lastly,
		
		\item  $\mathcal{M}$-open represents the case where the data generating mechanism cannot be conceptualized, and because of that cannot be included in $\mathcal{M}$.
	\end{itemize}
	
	\noindent While the partition of the model space into the above three classes is common in the Bayesian stacking literature, the definitions of $\mathcal{M}$-complete and $\mathcal{M}$-open problems are not entirely agreed upon. For examples of alternative (though very similar) definitions of the two, see \cite{BernardoSmith1994} or \cite{ClydeIversen2013}. The methods considered in the current study \citep[see e.g.][for jackknife averaging and Bayesian stacking respectively]{HansenRacine2012, Yao2018} are developed under the assumption that the candidate set does not contain a model encapsulating the data generating process. However, in order to construct asymptotic arguments, the data generating process needs to be conceptualized, and, as such, the working assumption when deriving oracle properties is $\mathcal{M}$-complete. 
	
	\subsection{Compromise estimation}
	
	In empirical applications, it is common to select one candidate model using some selection procedure, and proceed using the model as if it were the truth. As discussed above, this may yield sub-optimal performance, and in order to work around these issues, the modeling procedure can be altered to consider a compromise between all candidate models, rendering a combined prediction or posterior (predictive) distribution using information from the models in $\mathcal{M}$. The general form of the combined estimator or posterior probability is given by
	
	\begin{align}\label{eq:fma_stacking}
		\sum_{k = 1}^{K} w_k  \boldsymbol{\varphi}_k (\boldsymbol{X},  \boldsymbol{y}),
	\end{align}
	
	\noindent where $w_k$ is a weight attributed to the $k$:th candidate model and $\boldsymbol{\varphi}_k (\boldsymbol{X}, \boldsymbol{y})$ may be a frequentist prediction \citep[as in e.g.][]{HansenRacine2012}, a posterior predictive distribution \citep[as in e.g.][]{Yao2018}, or a Bayesian point prediction \cite[as in e.g.][]{ClydeIversen2013}, generated using the model $M_k$. While the presentation above concerns a very general form of compromise estimator, this paper specifically considers the cases where $\boldsymbol \varphi_k$ are either point predictions (or estimates) made using Bayesian linear or logistic regression models. The compromise weights, typically collected in the vector $\boldsymbol{w} = (w_1, \ldots, w_K)$, are assumed to belong to some set $\mathcal{W}$, which can be specified to restrict the estimated values $\hat{\boldsymbol{w}}$. In the literature, the most common assumption is that $\boldsymbol w$ lies in the unit simplex,
	
	\begin{align}\label{eq:weights_def}
		\boldsymbol w \in \mathcal{W} = \left\{ \boldsymbol{w}: \boldsymbol{w} \in [0, 1]^K, \ \ \sum_k w_k = 1 \right\}.
	\end{align}
	
	\subsection{Jackknife model averaging} \label{sec:JMA}
	
	Jackknife model averaging is introduced by \cite{HansenRacine2012}, who use LOO-CV to select model average weights that minimize a version of the least squares cross-validation criterion, given in \eqref{eq:ls_cv_crit}. 
	
	Let $\boldsymbol{X}_{-i}$ and $\boldsymbol{y}_{-i}$ denote the the matrix $\boldsymbol{X}$ and $\boldsymbol y$ with the $i$:th observation removed. Using $\boldsymbol{X}_{-i}$, $\boldsymbol{y}_{-i}$, and model $M_k$, compute the leave-one-out prediction of $y_i$ and denote it by $\Tilde{y}_{i,k}$. Further, define the vector $\Tilde{\boldsymbol{y}}_k = (\Tilde{y}_{1,k}, \ldots, \Tilde{y}_{n,K})^T$, the jackknife residual vector of the $k$:th candidate model, $\Tilde{\boldsymbol{e}}_k = \boldsymbol{y} - \Tilde{\boldsymbol{y}}_k$, and consider a vector of weights as given in \eqref{eq:weights_def}. Then, $\sum_{k = 1}^{K} w_k \Tilde{\boldsymbol{y}}_k$ is the averaged jackknife prediction of $\boldsymbol{y}$, which gives
	
	\begin{align}
		\boldsymbol{y} - \sum_{k = 1}^{K} w_k \Tilde{\boldsymbol{y}}_k
		= \sum_{k = 1}^{K} w_k (\boldsymbol{y} - \Tilde{\boldsymbol{y}}_k) 
		= \sum_{k = 1}^{K} w_k \Tilde{\boldsymbol{e}}_k
		= \tilde{\boldsymbol{e}} \boldsymbol{w}
	\end{align}
	
	\noindent as the jackknife average residual. Here $\Tilde{\boldsymbol{e}} = (\Tilde{\boldsymbol{e}}_1 \; \cdots \;  \Tilde{\boldsymbol{e}}_K )$ is an $n \times K$ matrix storing the jackknife residuals from each model column-wise. To estimate the jackknife average weights, define the squared error of the averaged jackknife prediction
	
	\begin{align} \label{eq:ls_cv_crit}
		\left( \boldsymbol{y} - \sum_{k = 1}^{K} w_k \Tilde{\boldsymbol{\mu}}_k \right)^{T} \left( \boldsymbol{y} - \sum_{k = 1}^{K} w_k \Tilde{\boldsymbol{\mu}}_k \right)
		= (\tilde{\boldsymbol{e}} \boldsymbol{w})^{T} (\tilde{\boldsymbol{e}} \boldsymbol{w}) 
		= \boldsymbol{w}^T \boldsymbol{S} \boldsymbol{w}, 
	\end{align}
	
	\noindent where $\boldsymbol{S} = \Tilde{\boldsymbol{e}}^{T} \Tilde{\boldsymbol{e}}$ is known as the least squares cross-validation criterion. Then, the the jackknife model average weights are given by
	
	\begin{align} \label{eq:jma_weights}
		\widehat{\boldsymbol{w}}_{\scriptscriptstyle \text{JMA}} 
		= \operatorname*{argmin}_{\boldsymbol{w} \in \mathcal{W}} \boldsymbol{w}^T \boldsymbol{S} \boldsymbol{w},
	\end{align}
	
	\noindent and the JMA prediction of $\boldsymbol y$ is computed by using $\hat{\boldsymbol y}_k$ and the elements of $\widehat{\boldsymbol{w}}_{\scriptscriptstyle \text{JMA}}$ in place of $w_k$ and $\boldsymbol \varphi_k$ in Equation \eqref{eq:fma_stacking}.
	
	\subsection{Bayesian stacking} \label{sec:BS}
	
	The fundamental idea of stacking is introduced by \cite{Wolpert1992}, but the method is first considered in a Bayesian setting by \cite{ClydeIversen2013}, who propose the use of cross-validation to estimate weights in a procedure that minimizes the prediction error of a stack of posterior predictive means. Bayesian stacking is developed further by \cite{Yao2018}, who generalize the idea by considering proper scoring rules as weight estimation criteria, and thus allow for stacking of predictive distributions as well as posterior predictive means. The approach of \cite{Yao2018} includes that of \cite{ClydeIversen2013} as a special case, where the latter is obtained by using a particular specification of the energy scoring rule.
	
	For a general treatment, let $\mathcal S$ denote a scoring rule (see \cite{Yao2018} for a formal definition). Then, the primary objective of Bayesian stacking is to select a vector of weights that is optimal according to a criterion defined using $\mathcal{S}$. Formally, the Bayesian stacking problem can be characterized as
	
	\begin{align} \label{eq:stack_general}
		\operatorname*{argmax}_{\boldsymbol{w} \in \mathcal{W}} \ \mathcal{S} 
		\left\{\sum_{k = 1}^{K} w_k p(y^*|\boldsymbol{y}, \boldsymbol{X}, \boldsymbol{x}^*, M_k), \ p_{\text{true}}(y^*|\boldsymbol{y}, \boldsymbol{X}, \boldsymbol{x}^*)\right\},
	\end{align}
	
	\noindent where $y^*$ is an unseen realization of $Y$, $p(\cdot|\boldsymbol{y}, \boldsymbol{X}, \boldsymbol x^*, M_k)$ is the posterior predictive density derived using the $k$:th candidate model, and $p_{\text{true}}(\cdot|\boldsymbol{y}, \boldsymbol{X}, \boldsymbol x^*)$ is the data generating distribution. The subject of the stacking procedure is determined by the choice of $\mathcal S$, and \cite{Yao2018} provide grounds for using the logarithmic scoring rule to stack predictive distributions, whereas \cite{ClydeIversen2013} motivate the use of the energy scoring rule to combine posterior predictive means.
	
	Consider the proposal of \cite{ClydeIversen2013}, where $\mathcal S = \frac{1}{2} \mathbb{E}_P \Vert Y - Y'\Vert^\beta - \mathbb{E}_P \Vert Y - y\Vert^\beta$ is the energy score. Here, $Y$ and $Y'$ are independent random variables that follow some common distribution $P$, $y$ is a realization of $Y$, and $\beta$ is a constant. When $\beta = 2$, the energy score reduces to  $-\Vert \mathbb{E}_P [Y] - y\Vert^2$. Thus, when $P$ is taken as the stacked distribution given in \eqref{eq:stack_general}, the corresponding optimization problem can be stated as
	
	\begin{eqnarray} \label{eq:stacking_means}
		& & \operatorname*{argmax}_{\boldsymbol{w} \in \mathcal{W}} \left\{-\left\lVert \sum_k w_k \mathbb E[y^* | \boldsymbol{y}, \boldsymbol{X}, \boldsymbol{x}^*, M_k] - y^* \right \rVert^2 \right\} \nonumber \\
		&=& \operatorname*{argmin}_{\boldsymbol{w} \in \mathcal{W}} \left \lVert \sum_k w_k \mathbb E[y^*|\boldsymbol{y}, \boldsymbol{X}, \boldsymbol{x}^*,  M_k] - y^* \right \rVert^2,
	\end{eqnarray}
	
	\noindent where 
	
	\begin{align}
		\mathbb E[y^*|\boldsymbol{y}, \boldsymbol{X}, M_k] = \int y^* p(y^* | \boldsymbol y, \boldsymbol X, \boldsymbol x^*, M_k) d y^*
	\end{align}
	is the mean of the $k$:th candidate posterior predictive distribution. Since $y^*$ is unseen, the expressions in \eqref{eq:stacking_means} are intractable. However, given some assumptions,  \cite{LeClarke2017} show that
	
	\begin{eqnarray*}
		&& \frac{1}{n} \sum_{i = 1}^{n}  \left \lVert \sum_k w_k \mathbb E[y_i|\boldsymbol{y}_{-i}, \boldsymbol{X}_{-i}, M_k] - y_i \right \rVert^2 \\
		&& \hspace{0.5cm} - \int  \left \lVert \sum_k w_k \mathbb E[y^*|\boldsymbol{y}, \boldsymbol{X}, \boldsymbol{x}^*,  M_k] - y^* \right \rVert^2 p(y^* | \boldsymbol y, \boldsymbol X, \boldsymbol x^*, M_k) d y^* 
		\stackrel{L_2}{\longrightarrow} 0,
	\end{eqnarray*}
	
	\noindent as $n \rightarrow \infty$, under the posterior marginal distribution of the outcome $\boldsymbol{y}$. This motivates, leave-one-out cross-validation as a way of approximating the optimization problem from a Bayesian perspective. As such, the weights for stacking posterior predictive means are estimated by
	
	\begin{align} \label{eq:stack_crit}
		\hat{\boldsymbol w}_{\scriptscriptstyle \text{BS}}  = \operatorname*{argmin}_{\boldsymbol{w} \in \mathcal{W}} \frac{1}{n} \sum_{i = 1}^{n} \left \lVert \sum_k w_k \mathbb E[y_i|\boldsymbol{y}_{-i}, \boldsymbol{X}_{-i}, M_k] - y_i \right \rVert^2.
	\end{align}
	
	\noindent Here, it is clear that stacking of posterior means, as proposed by \cite{ClydeIversen2013}, closely resembles jackknife model averaging. Thus, certain optimality properties derived for jackknife model averages can be extended to consider Bayesian stacking estimators. Such extensions are considered below, for linear and logistic regression models.
	
		\section{Asymptotic optimality} \label{sec:optimality}
	
	As mentioned, the weighting schemes presented in the preceding section are similar in the sense that they both rely on leave-one-out cross-validation to select compromise weights, albeit one combines candidates estimated in the frequentist paradigm while the other considers candidates motivated from a Bayesian perspective. In light of this, it is pertinent to determine whether the (frequentist) asymptotic properties of JMA extend to Bayesian stacking. One aspect of interest is the oracle property,
	
	\begin{align} \label{eq:oracle}
		\frac{L_n(\hat{\boldsymbol{w}})}{\inf_{\boldsymbol{w} \in \mathcal{W}} L_n(\boldsymbol{w})} \stackrel{p}{\longrightarrow} 1, \text{ as } n \rightarrow \infty
	\end{align}
	
	\noindent where $L(\boldsymbol{w})$ is some loss function. The oracle property states that asymptotically, the estimated weight vector $\hat{\boldsymbol w}$ yields the same loss as the weight vector that minimzes the loss over the specific sets $\mathcal M$ and $\mathcal W$. In particular, when $\mathcal W$ is the unit simplex given in \eqref{eq:weights_def}, the oracle implies that the compromise estimator based on $\hat{\boldsymbol w}$ asymptotically performs at least as good as any one of the single candidates in $\mathcal M$.
	
	The oracle property has been proven in many applications of frequentist model averagin (FMA). For example, \cite{HansenRacine2012}, \cite{ZhangWanZhou2013}, and \cite{Zhao2020} prove oracle properties of average estimators combining linear regression models under different circumstances, while \cite{ZhaoZhouYang2019} and \cite{ZuljJin2021} derive similar results for generalized linear models. Proving such properties for Bayesian stacking of means can be considered as a frequentist motivation for using the method, but also as an initial step towards establishing deeper connections between Bayesian and frequentist methods. In the sections that follow, Theorem \ref{thm:linear_oracle} establishes an oracle property for a compromise between types of Bayesian linear regression models, while a similar oracle for Bayesian logistic regression models is discussed in Section \ref{sec:logitoracle}.
	
	\subsection{Linear models} \label{sec:linregoracle}
	
	In the same way as \cite{Zhao2020}, suppose the outcome $y_i$ is generated by
	
	\begin{align} \label{eq:linreg_dgp}
		y_i = \sum_{j = 1}^{\infty} z_{ij}\gamma_j + \varepsilon_i,
	\end{align}
	
	\noindent where $\boldsymbol{z}_i = (z_{i1}, z_{i2}, \ldots)^T$ is of countably infinite dimension, $\boldsymbol{Z} = (\boldsymbol{z}_1 \; \cdots \; \boldsymbol{z}_n)^T$, $\boldsymbol{\varepsilon}$ is an $n \times 1$ vector of random error terms that satisfies $\mathbb{E}(\varepsilon | \boldsymbol{Z}) = \boldsymbol{0}_{n \times 1}$, and $\mathbb{E}(\boldsymbol{\varepsilon} \boldsymbol{\varepsilon}^T | \boldsymbol{Z}) = \boldsymbol{\Omega} = \operatorname{diag}(\sigma^2_1, \ldots, \sigma^2_n)$. Here, like in \cite{Hansen2007} and \cite{Zhao2020}, a countably infinite number of $z_{i, j}$ is considered to reflect the fact that, in practice, the data generating process is often involved and may include a vast amount different factors. This is a situation where, typically, the data generating mechanism can be used to formulate probabilistic arguments,  but not realistically be included in the candidate set, thus emulating the $\mathcal{M}$-complete scenario. 
	
	Now, consider modeling the outcome using the data matrix $\boldsymbol{X}_{n \times p}$, where $\boldsymbol{X}$ might include a subset of the covariates in $\boldsymbol{Z}$ or some functions of certain covariates in $\boldsymbol{Z}$, but also might include some variables that are not necessarily present in the data generating process. Note that $\boldsymbol{X}$ has $p$ columns, meaning that the dimension of $\boldsymbol X$ may be different from the dimension of $\boldsymbol Z_{n \times q}$. In particular, consider the case where predictions are made as $\boldsymbol{X\hat{\beta}}_\lambda$, where $\boldsymbol{\hat{\beta}}_\lambda$ is the ridge regression estimator
	
	\begin{align}
		\hat{\boldsymbol{\beta}}_\lambda = \left(\boldsymbol{X}^T\boldsymbol{X} + \lambda \boldsymbol{I}_{p \times p}\right)^{-1} \boldsymbol{X}^T \boldsymbol{y},
	\end{align}
	
	\noindent with $\lambda \geq 0$. In general, the value of $\lambda$ is not given, but needs to be determined. This is typically done by comparison of a set of feasible values, $\{\lambda_1, \ldots, \lambda_K\}$, often using some form of out-of-sample validation. However, instead of selecting a single value of $\lambda$, each $\lambda_k$ can be seen as representing a candidate model. Then, the methods discussed earlier can be used to form a compromise between different $\lambda_k$, by setting  $\varphi_k(\boldsymbol{X}, \boldsymbol{y}) = \boldsymbol{X}\hat{\boldsymbol{\beta}}_{\lambda_k}$ in Equation \eqref{eq:fma_stacking}. Under certain conditions, \cite{Zhao2020} show that the jackknife average estimator combining such ridge regressions has the oracle property \eqref{eq:oracle}, with respect to $L(\boldsymbol{w}) = \lVert \boldsymbol{\mu} - \sum_k w_k \boldsymbol{X} \hat{\boldsymbol{\beta}}_{\lambda_k} \rVert^2$, where $\boldsymbol{\mu} = \mathbb{E}(\boldsymbol{y} | \boldsymbol{Z})$. 
	
	In a Bayesian regression model assuming $y_i|\boldsymbol{x}_i, \boldsymbol{\beta} \sim \mathcal{N}(\boldsymbol{x}_i^T \boldsymbol{\beta}, \sigma^2)$ with $\sigma^2$ taken as known, and using $\mathcal{N}(\boldsymbol{0}_{p \times 1}, \gamma^2 \boldsymbol{I}_{p \times p})$ as the prior distribution for $\boldsymbol{\beta}$, the posterior mean (or maximum a posteriori estimator) of $\boldsymbol{\beta}$ is $( \boldsymbol{X}^T\boldsymbol{X} + (\sigma/\gamma)^{2} \boldsymbol{I}_{p \times p} )^{-1} \boldsymbol{X}^T \boldsymbol{y}$. Thus, when $\gamma$ and $\sigma$ are chosen suitably, the point predictions made using this Bayesian estimator are equivalent to those made using the ridge regression model, and the criterion used to estimate Bayesian stacking weights in \eqref{eq:stack_crit} is equivalent to the jackknife weighting criterion in \eqref{eq:jma_weights}. As a result, the oracle property of \cite{Zhao2020} extends to Bayesian stacking of linear regressions, where candidates are distinguished through their priors, $\boldsymbol{\beta} \sim \mathcal{N}(\boldsymbol{0}_{p \times 1}, \gamma_k^2 \boldsymbol{I}_{p \times p})$ for $k = 1, \ldots, K$. In this context, it is interesting to further consider stacking of Bayesian regressions models assuming $\mathcal{N}(\boldsymbol{0}_{p \times 1}, \boldsymbol{S}_k)$ as priors for $\boldsymbol{\beta}$, where $\boldsymbol{S}_k > 0$ are symmetric but not necessarily diagonal. Such priors include, for example, the $g$ prior introduced by \cite{Zellner1986}, where $\boldsymbol{S}_k = g_k (\boldsymbol{X}^T\boldsymbol{X})^{-1}$, and $g_k > 0$ is some constant.
	
	Consider a set of $K$ Bayesian linear regressions with $y_i | \boldsymbol{x}_i, \boldsymbol{\beta}, M_k \sim \mathcal{N}(\boldsymbol{x}_i^T \boldsymbol{\beta}, \sigma^2)$, $\boldsymbol{\beta} | M_k \sim \mathcal{N} (\boldsymbol{0}_{p \times 1}, \boldsymbol{S}_{k})$, assuming $\sigma^2$ known and taking $\boldsymbol{S}_k$ as symmetric and positive definite. Then, $( \boldsymbol{X}^T\boldsymbol{X} + \sigma^{2} \boldsymbol{S}_k^{-1} )^{-1} \boldsymbol{X}^T \boldsymbol{y}$ is the posterior mean of $\boldsymbol{\beta}$ under model $M_k$, and the corresponding posterior predictive mean is given by 
	
	\begin{align}
		\mathbb E[y_i|\boldsymbol{y}_{-i}, \boldsymbol{X}_{-i}, M_k]  = \boldsymbol{x}_{i}^T  \left( \boldsymbol{X}_{-i}^T\boldsymbol{X}_{-i} + \sigma^{2} \boldsymbol{S}_k^{-1} \right)^{-1} \boldsymbol{X}_{-i}^T \boldsymbol{y}_{-i}.
	\end{align}
	
	\noindent Using this, the weights for stacking posterior predictive means are estimated by
	
	\begin{align}
		\hat{\boldsymbol{w}}_{\scriptscriptstyle \text{BS}}  = \operatorname*{argmin}_{\boldsymbol{w} \in \mathcal{W}} \frac{1}{n} \sum_{i = 1}^{n} \left \lVert \sum_k w_k \mathbb E[y_i|\boldsymbol{y}_{-i}, \boldsymbol{X}_{-i}, M_k] - y_i \right \rVert^2.
	\end{align}
	
	\noindent Theorem \ref{thm:linear_oracle} establishes the oracle property of $\sum_k \hat{w}_{\scriptscriptstyle \text{BS}, k} \mathbb E[\boldsymbol y|\boldsymbol{y}, \boldsymbol{X}, M_k]$, with
	
	\begin{align} \label{eq:postpred_kean}
		\mathbb E[\boldsymbol y|\boldsymbol{y}, \boldsymbol{X}, M_k] = \int \boldsymbol y \; p(\boldsymbol y | \boldsymbol y, \boldsymbol X, M_k) d \boldsymbol y = \boldsymbol X \left( \boldsymbol{X}^T\boldsymbol{X} + \sigma^{2} \boldsymbol{S}_k^{-1} \right)^{-1} \boldsymbol{X}^T \boldsymbol{y}.
	\end{align}
	
	To present the assumptions required to prove the theorem, some notation is introduced. First, let $\lambda_{\text{min}}(\boldsymbol{A})$ and $\lambda_{\text{max}}(\boldsymbol{A})$ denote the smallest and largest eigenvalues of the matrix $\boldsymbol{A}$, respectively. Further, define $\boldsymbol{P}_k = \boldsymbol{X} ( \boldsymbol{X}^T\boldsymbol{X} + \sigma^{2} \boldsymbol{S}_k^{-1})^{-1} \boldsymbol{X}^T$, and let $R(\boldsymbol{w}) = \mathbb{E}[L(\boldsymbol{w})| \boldsymbol{Z}] = \mathbb{E}\left[\lVert \boldsymbol{\mu} - \bar{\boldsymbol{\mu}} \rVert^2| \boldsymbol{Z}\right]$, and, subsequently, $\xi_n = \inf_{\boldsymbol{w} \in \mathcal{W}} R(\boldsymbol{w})$. Finally, let $\boldsymbol{w}_k^0$ be a unit vector where the $k$:th element is a one and the remaining elements are all zero. Unless stated otherwise, all limits are taken as $n \rightarrow \infty$.
	
	\begin{assumption} \label{ass:lm1} 
		$\lambda_{\text{min}}(\boldsymbol{X}^T \boldsymbol{X} / n)$ and $\lambda_{\text{max}}(\boldsymbol{X}^T \boldsymbol{X} / n)$ are bounded, below and above, by constants $c_0, c_1 \geq 0$, respectively. Also, $n^{-1/2}\boldsymbol{X}^T \boldsymbol{\varepsilon} = O_p(1)$. 
	\end{assumption}
	
	\begin{assumption} \label{ass:lm2} 
		$p^* = O\left(n^{-1}\right)$. Here, $p^* = \max_{1 \leq k \leq K} \max_{1 \leq i \leq n} P^k_{ii}$, and $P^k_{ii}$ is the $i$:th diagonal element of $\boldsymbol{P}_k$.
	\end{assumption}
	
	\begin{assumption} \label{ass:lm3} 
		$\sup_{i \in \{1, \ldots, n\} } \mathbb{E} [\varepsilon_i^4|\boldsymbol{z}_i] = O(1)$, almost surely.
	\end{assumption}
	
	\begin{assumption} \label{ass:lm4} 
		$\boldsymbol{\mu}^T \boldsymbol{\mu} / n = O(1)$, almost surely.
	\end{assumption}
	
	\begin{assumption} \label{ass:lm5} 
		$\xi_n^{-2} \sum_k R(\boldsymbol{w}_k^0) = o(1)$, almost surely.
	\end{assumption}
	
	\noindent The assumptions are the same as those made by \cite{Zhao2020}, who provide one-by-one discussion of each assumption.
	
	\begin{theorem} \label{thm:linear_oracle}
		Suppose that $\hat{\boldsymbol{w}}_{\scriptscriptstyle \text{BS}}$ is used to compute $\bar{\boldsymbol{\mu}} = \sum_k \hat{w}_k \mathbb E[\boldsymbol y|\boldsymbol{y}, \boldsymbol{X}, M_k] $, with $\mathbb E[\boldsymbol y|\boldsymbol{y}, \boldsymbol{X}, M_k] $ as given in Equation \eqref{eq:postpred_kean}. Then, given assumptions \ref{ass:lm1}--\ref{ass:lm5},  $\bar{\boldsymbol{\mu}}$ has the oracle property \eqref{eq:oracle} with respect to $L(\boldsymbol{w}) = \lVert \boldsymbol{\mu} - \bar{\boldsymbol{\mu}} \rVert^2$.
	\end{theorem}
	
	\begin{proof}
		The proof is given in Appendix \ref{app:lin_oracle}.
	\end{proof}
	
	Theorem \ref{thm:linear_oracle} states that, under the conditions specified, the model average computed using the Bayesian stacking weights will, asymptotically, perform as well as any other compromise comprising the candidate models and using weights from the same simplex. As stated in the theorem, performance is measured in terms of $L(\boldsymbol{w}) = \lVert \boldsymbol{\mu} - \bar{\boldsymbol{\mu}} \rVert^2$.
	\begin{remark}
		For simplicity, the error variance $\sigma^2$ in the normal likelihood $y_i|\boldsymbol{x}_i, \boldsymbol{\beta} \sim \mathcal{N}(\boldsymbol{x}_i^T \boldsymbol{\beta}, \sigma^2)$ is assumed to be known. However, if $\sigma^2$ is assumed unknown, the posterior expectation of $\boldsymbol \beta$, conditional on $\sigma^2$, is given by $\mathbb E [\boldsymbol \beta | \boldsymbol y, \boldsymbol X, \sigma^2, M_k] = ( \boldsymbol{X}^T\boldsymbol{X} + \sigma^{2} \boldsymbol{S}_k^{-1} )^{-1} \boldsymbol{X}^T \boldsymbol{y}$, which is the same as the posterior mean given above. Hence, the theorem applies when $\sigma^2$ is assumed unknown, but posterior means are derived conditional on $\sigma^2$. 
	\end{remark}

	\begin{remark}
		The theorem allows for $\boldsymbol S_k$ to depend to the sample size $n$, as long as the resulting matrix $\boldsymbol P_k$ satisfies assumption \ref{ass:lm2}. This is the case for the $g$ prior of \cite{Zellner1986}.
	\end{remark}
	
	\subsection{Logistic models} \label{sec:logitoracle}
	
	Suppose that $Y_i \in \{0, 1\}$, with $P(Y_i = 1) = p_i$, and that a model for the probabilities $p_i$ is required. Let $\hat{\boldsymbol p }_k = (\hat{p}_{1, k}, \ldots, \hat{p}_{n, k})^T$ denote a candidate estimate of $\boldsymbol p = (p_1, \ldots, p_n)^T$, and define the corresponding compromise estimator as $\bar{\boldsymbol p } = \sum_k \hat{w}_k \hat{\boldsymbol{p}}_k$. When $\hat p_{i, k} = \operatorname*{expit}(\boldsymbol x_i^T \hat{\boldsymbol \beta}_k)$ are outputs of frequentist logistic regression models, compromises such as $\bar{\boldsymbol p}$ are considered by e.g. \citet{ZhaoZhouYang2019} and \citet{ZuljJin2021}, who establish oracle properties of compromises where the weights are estimated using cross-validation criteria of the form
	
	\begin{align} \label{eq:logit_weight}
		\hat{\boldsymbol{w}}  
		= \operatorname*{argmin}_{\boldsymbol{w} \in \mathcal{W}} \frac{1}{n} \sum_{i = 1}^{n} \left \lVert \sum_k w_k \operatorname{expit}(\boldsymbol{x}_{i}^T \hat{\boldsymbol{\beta}}_{-i, k})  - y_i \right \rVert^2,
	\end{align}
	
	\noindent where $\hat{\boldsymbol \beta}_{-i, m}$ is the leave-one-out coefficient estimate made using a candidate logistic regression model. In particular, Theorem 1 of \citet{ZhaoZhouYang2019} provides an oracle property that asserts asymptotic optimality with respect to the squared loss function, $L(\boldsymbol{w}) = \lVert \boldsymbol{p} - \bar{\boldsymbol{p}} \rVert^2$. The theorem is valid given Assumptions \ref{ass:logit1}-\ref{ass:logit3}, listed below.
	
	\begin{assumption} \label{ass:logit1} 
		For candidate model $k$, there exists a limit $\boldsymbol{\beta}^*_k$ such that $\hat{\boldsymbol{\beta}}_k - \boldsymbol{\beta}^*_k = O_p(n^{-1/2})$.
	\end{assumption}
	
	\noindent Before stating the remaining assumptions,  let $\boldsymbol{p}^*_k = \hat{\boldsymbol{p}}_k |_{\hat{\boldsymbol{\beta}}_k = \boldsymbol{\beta}^*_k}$, $\bar{\boldsymbol{p}}^* = \sum_k w_k \boldsymbol{p}^*_k$, and  $L^*(\boldsymbol{w}) = \lVert \bar{\boldsymbol{p}}^* - \boldsymbol{p} \rVert^2$, with $\xi_n = \inf_{\boldsymbol{w} \in \mathcal{W}} L^*(\boldsymbol{w})$. Then, the last two assumptions are as follows.
	
	\begin{assumption} \label{ass:logit2} 
		$\xi^{-1}_n n^{1/2} = o(1)$. 
	\end{assumption}
	
	\begin{assumption} \label{ass:logit3} 
		For any candidate, $\partial \hat{p}_{i, k} / \partial \hat{\boldsymbol{\beta}}_k |_{\hat{\boldsymbol{\beta}}_k = \tilde{\boldsymbol{\beta}}_{k, i}} = O_p(1)$, uniformly for $i \in \{1, \ldots, n\}$ and for any $\tilde{\boldsymbol{\beta}}_{k, i}$ that lies between $\hat{\boldsymbol{\beta}}_k$ and $\boldsymbol{\beta}^*_k$. 
	\end{assumption}

	Consider applying the oracle property of \citet{ZhaoZhouYang2019} to a compromise between Bayesian candidates. Bayesian logistic regression models assume $p_i = \operatorname{expit}( \boldsymbol x_i^T \boldsymbol \beta )$ and $y_i | \boldsymbol x_i, \boldsymbol \beta \sim \text{Bernoulli}(p_i)$, along with a prior $p(\boldsymbol \beta)$ on the vector of regression coefficients. Now, let $\mathbb E[\boldsymbol \beta | \boldsymbol y, \boldsymbol X, M_k]$ denote the posterior mean of $\boldsymbol \beta$, derived using candidate model $M_k$. Then,  $\hat{p}_{i, k} = \operatorname{expit}(\boldsymbol x_i^T \mathbb E[\boldsymbol \beta | \boldsymbol y, \boldsymbol X, M_k])$ give Bayesian point estimates of $p_i$, and a compromise $\bar{\boldsymbol p}$ can be formed accordingly. The weighting criterion in \eqref{eq:logit_weight} can easily be rewritten to consider such Bayesian candidates, replacing $\hat{\boldsymbol{\beta}}_{-i, m}$ by $\mathbb E[\boldsymbol \beta | \boldsymbol y_{-i}, \boldsymbol X_{-i}, M_k]$. Thus, under the assumptions given above, the oracle property extends to this type of Bayesian stacking. \citet{ZhaoZhouYang2019} discuss the feasibility of the assumptions with reference to frequentist discrete choice models, but the assumptions are yet to be considered from the perspective of Bayesian stacking. Hence, the remainder of this section aims to discuss the assumptions, with a view to motivate applications of the oracle property to Bayesian stacking.
	
	The assumptions are considered standard in the literature on FMA, and some discussion is provided by \cite{ZhaoZhouYang2019}. All three assumptions concern asymptotic properties, meaning that knowledge of the true data generating mechanism is needed to evaluate them, making it difficult to verify the assumptions in practice. However, Assumption \ref{ass:logit2} requires the squared error of the infeasible best compromise to diverge at a rate faster than $\sqrt n$. According to \cite{ZhaoZhouYang2019}, the assumption generally holds for frequentist candidates that are misspecified, and since Bayesian stacking primarily considers $\mathcal M$ complete/open problems, the assumption is considered reasonable. Assumption \ref{ass:logit3}, on the other hand, restricts the derivatives of the estimated probabilities, assuming that they are bounded uniformly. In logistic regression, the derivative of $\hat{\boldsymbol p}$ depends on the data matrix $\boldsymbol X$, and $\hat{\boldsymbol p}$ itself. Since the analysis is done conditional on $\boldsymbol X$, $\hat{\boldsymbol p}$ is the only random component of the derivative, and since $\hat{\boldsymbol p}$ is assumed to converge in Assumption \ref{ass:logit1}, Assumption \ref{ass:logit3} is considered reasonable.
	
	Assumption \ref{ass:logit1} requires the $k$:th model point estimate of $\boldsymbol \beta$ to converge, at a specified rate, to some limit $\boldsymbol \beta^*_k$. The assumption is reasonable when the candidates are frequentist logistic regression models, since, in general, maximum likelihood estimators converge to the parameter vector that minimizes the Kullback-Leibler divergence from the data generating distribution, as discussed by \citet{Akaike1973} and \citet{White1982}. \citet{AndoLi2017} prove the existence of $\boldsymbol{\beta}^*_k$ for high-dimensional generalized linear models. Thus, $\boldsymbol{\beta}^*_k$ may be taken as this limit when the candidates are estimated using maximum likelihood. However, considered from the perspective of Bayesian stacking, the assumption requires $\mathbb E[\boldsymbol \beta | \boldsymbol y, \boldsymbol X, M_k]$ to converge, in the data generating probability and at the specified rate, to a limit defined in a similar way. To date, there are no proofs of such properties for misspecified models with independent but not identically distributed data in the Bayesian setting, but nonetheless, there are theoretical results pointing in that general direction. These are discussed in the paragraphs that follow.
	
	Theorem 9 of \citet{Ramamoorthi2015} asserts that, for each model $M_k$, the posterior distribution of $\boldsymbol{\beta}_k$ is asymptotically carried by an arbitrarily small neighborhood of the limiting value given by \citet{White1982}. However, while the theorem establishes posterior concentration, it does not prove that convergence occurs at the rate required by Assumption \ref{ass:logit1}, and it also does not guarantee that the neighborhood over which the posterior concentrates contains only one point, $\boldsymbol{\beta}^*_k$. Addressing the first issue, \citet{KleijnvanderVaart2012} show in their Lemma 2.1 that, with misspecified models and with $Y_i$ i.i.d., the sequence of posterior distributions converges to a normal distribution at a $\sqrt{n}$ rate given certain assumptions. In the current case, $Y_i$ are independent \emph{but not} identically distributed, meaning that the lemma does not apply. However, it may be taken as an indication that \ref{ass:logit1} is not unfeasible. 
	
	Theorem 9 of \citet{Ramamoorthi2015} is valid under the assumptions presented in Appendix \ref{app:concentration}. The theorem does not necessarily require normal priors for $\boldsymbol \beta$, but is applicable to any prior for which the assumptions are satisfied. Appendix \ref{app:concentration} also includes a section where the conditions of the theorem are verified for the logistic regression model assuming either $\boldsymbol \beta \sim \mathcal N(\boldsymbol 0, \boldsymbol \Sigma)$ or $\boldsymbol \beta \sim \mathcal T(\nu, \boldsymbol \Sigma)$, where the priors have been identified due to their convenience and popularity in practice. The combination of this convergence result and the stability of estimates assumed in \ref{ass:logit3} provide grounds for investigating Bayesian stacking of such candidates in a set of Monte Carlo simulation studies, presented below.
	
	\begin{remark}
		Even though Bayesian stacking was presented as a way of combining posterior means, the above theorem does not require $\hat{\boldsymbol{p}}_k$ to be based on the posterior mean of $\boldsymbol \beta$. In fact, any Bayesian point estimator that satisfies the conditions of the theorem may be used to compute $\hat{\boldsymbol{p}}_k$. As such, the theorem may also apply in situations where, for example, predictions based on maximum aposteriori estimates are combined.
	\end{remark}

	\begin{remark}
		Theorem \ref{thm:linear_oracle} relies on the closed form posterior mean that follows from the assumption that $\boldsymbol \beta \sim \mathcal N (\boldsymbol 0, \boldsymbol S)$. However, the literature on frequentist model averaging also provides oracle properties that do not require a closed form of $\mathbb E[\boldsymbol \beta | \boldsymbol y, \boldsymbol X]$, but only need candidate estimates of $\boldsymbol \beta$ to have well defined limits, and to converge at a certain rate \citep[see e.g.][]{ZuljJin2021}. The formal statements of these convergence conditions are often similar to Assumption \ref{ass:logit1}, meaning that the discussion of the preceding paragraphs may be applied to linear regression models as well. Thus, oracle properties derived under such conditions may well apply to stacking of linear regression models assuming non-normal priors. This paper considers the use of $\mathcal T$ priors in simulation, but does not provide proofs of oracles that apply in such cases. 
	\end{remark}
	
		\section{Monte Carlo experiments} \label{sec:simulations}
	
	The oracle properties stated above indicate that Bayesian stacking is bound to perform at least as well as any single candidate model in $\mathcal M$, and the simulations that follow aim to study the application and benefit of stacking both linear and logistic regression models estimated using finite samples. Throughout the simulation study, the predictive performance of Bayesian stacking is compared to that of the best (in terms of cross-validation error) candidate model in $\mathcal M$. As such, the results of the study should not be read as a comparison of different ways to implement Bayesian stacking, but rather as an example of how Bayesian stacking compares to model selection by cross-validation, which is in line with the oracle properties presented earlier.
	
	\subsection{Linear regression}
	In the linear regression simulations, outcomes  $y_i$ are generated as $Y_i | \boldsymbol{x}_i \sim \mathcal{N}(\boldsymbol{x}_i^T \boldsymbol{\beta}, \sigma^2)$, for $i = 1, \ldots, n \in \{50, 100\}$. The vectors $\boldsymbol x_i$ are generated independently as standard normal variates with 1000 elements, and $\beta_j = c\sqrt{2}j^{-3/2}$, for $j = 1, \ldots, 1000$. Further, $\sigma^2 = 1$, and $c$ is a scalar used to control the theoretical $R^2$ of the generating mechanism. The values of the scaling constants $c$ are chosen so that
	\begin{align*}
		R^2 = \frac{\operatorname*{Var}[c\boldsymbol X \boldsymbol \beta]}{\operatorname*{Var}[c\boldsymbol X \boldsymbol \beta] + \sigma^2} \in \{0.2, \ldots, 0.8\}.
	\end{align*}

	The simulation replicates the design of \citet{Hansen2007}, aiming to mimic the data generating process presented in Equation \eqref{eq:linreg_dgp}. Here, the decaying nature of the coefficients $\beta_j$ is incorporated to simulate a situation where a few variables strongly influence the outcome, whereas a large amount of variables are only weakly associated with $y$. This is believed to reflect a realistic prediction problem, where the presence of many relevant but perhaps weakly associated variables have to be considered in model selection. Naturally, computational limitations prevent the use of infinitely many variables to generate outcomes, which is why the series in \eqref{eq:linreg_dgp} is truncated at $j = 1000$. However, only 14 of the relevant variables are used to fit the candidate models in $\mathcal M$, meaning that all candidates are misspecified by default.

	To study the performance of Bayesian stacking, stacked predictions are formed using two different sets of candidate models, $\mathcal{M}_g$ and $\mathcal{M}_T$. Here,
	\begin{itemize}
		\item $\mathcal{M}_g$ collects regression models assuming $\boldsymbol{\beta}_k \sim \mathcal{N}_p \left[\boldsymbol{0}, \; g_k ( \boldsymbol{X}^T \boldsymbol{X})^{-1}\right]$, with $g_k$ are equally spaced in $ [10^{-2}, \ 10^3],$ for $m = 1, \ldots, 100$, and
		
		\item $\mathcal{M}_T$ consists of candidate models where $\boldsymbol{\beta}_k \sim \mathcal{T}_p \left[ \nu_k, \; \lambda (\boldsymbol X^T \boldsymbol X)^{-1} \right]$, where the degrees of freedom $\nu_k \in \{1, 2, 3, 4, 5, 10, 30\}$, and $\lambda$ is a scaling parameter set to $2.5$ in Design 1 and $0.1$ in Design 2. Since there is no closed form of the posterior mean of $\boldsymbol \beta$, predictions are computed using the corresponding MAP estimates.
	\end{itemize}
	
	\noindent In each replication, a test set of 500 observations is generated and used to evaluate predictive performance, and compare the stacked prediction to the best candidate model. The best candidate, $M_*$, is selected as the model that produces the smallest cross-validation squared error. In all simulations, $R = 10 \ 000$ replications are made for each of the seven values of $c$, and the performance of the compromise estimators is measured as the ratio of the prediction errors of the compromise predictions versus the best model prediction, i.e. 
	\begin{align} \label{eq:ratio_linreg}
		r_\mu = \frac{\frac{1}{R} \sum_{r = 1}^R (\lVert \boldsymbol{X}_{\text{test}} \boldsymbol{\beta} - \sum_k \hat
			w_k \boldsymbol{X}_{\text{test}} \mathbb E[\boldsymbol \beta | \boldsymbol y, \boldsymbol X, M_k] \rVert^2)_r}{\frac{1}{R} \sum_{r = 1}^R (\lVert \boldsymbol{X}_{\text{test}} \boldsymbol{\beta} - \boldsymbol{X}_{\text{test}} \mathbb E[\boldsymbol \beta | \boldsymbol y, \boldsymbol X, M_*]  \rVert^2)_r}.
	\end{align}
	
	\noindent Here, $\boldsymbol{X}_{\text{test}} \mathbb E[\boldsymbol \beta | \boldsymbol y, \boldsymbol X, M_*]$ is the point prediction generated using the best candidate model in the corresponding set of candidates (i.e. either in $\mathcal{M}_g$ or $\mathcal{M}_T$).
	
	\subsubsection{Results} \label{sec:linear_results}
	The results of the linear regression simulations are given in the panels of Figure \ref{fig:linreg}, illustrating $r_\mu$ as a function of the theoretical signal-to-noise ratio. The top panel shows the relative performance of estimators stacking the candidate models in $\mathcal M _g$, while the right hand panel shows the corresponding performance of estimators stacking candidates from $\mathcal M_T$. As can be seen, the curves indicate that, on average, Bayesian stacking performs either better than or on par with model selection using LOO-CV. The averaging of candidate models looks to be particularly fruitful in smaller samples and when the data generating mechanism is noisy, while, on the other hand, the benefits of stacking are not as clear as $n$ increases or as the signal in the data grows stronger.

	\subsection{Logistic regression}
	
	Binary data are simulated as $Y_i \sim \text{Bernoulli}(p_i)$, where $p_i = \operatorname{expit}[c \boldsymbol{x}_i^T \boldsymbol{\beta}]$.	The $\boldsymbol{x}_i$ are sampled as 1000-dimensional vectors from a standard normal distribution, whereas $\boldsymbol{\beta}$ is made up of $\beta_j = (-1)^{j + 1} \sqrt{2}j^{-1.5}$, for $j = 1, \ldots, 1000$. This set-up is copied from \citet{ZuljJin2021}, who modify the DGP of \citet{Hansen2007} for better use with logistic regression (i.e. to avoid generation of too many corner probabilities $p_i \in \{0, 1\}$). As above, the value of $c$ is varied to control the signal-to-noise ratio according to
	
	\begin{align*}
		R^2 = \frac{\operatorname*{Var}[c\boldsymbol{X\beta}]}{\operatorname*{Var}[c\boldsymbol{X\beta}] + \frac{\pi^2}{3}} \in \{0.2, \ldots, 0.7\},
	\end{align*}
	
	\noindent Here, $Y$ is assumed to be a discretization of $Y^*$, a continuous random variable following the standard logistic distribution with variance $\pi^2/3$. 
	
		\begin{figure}[t!]
		\centering
		\begin{subfigure}{0.5\textwidth}
			\centering
			\includegraphics[scale = 0.35]{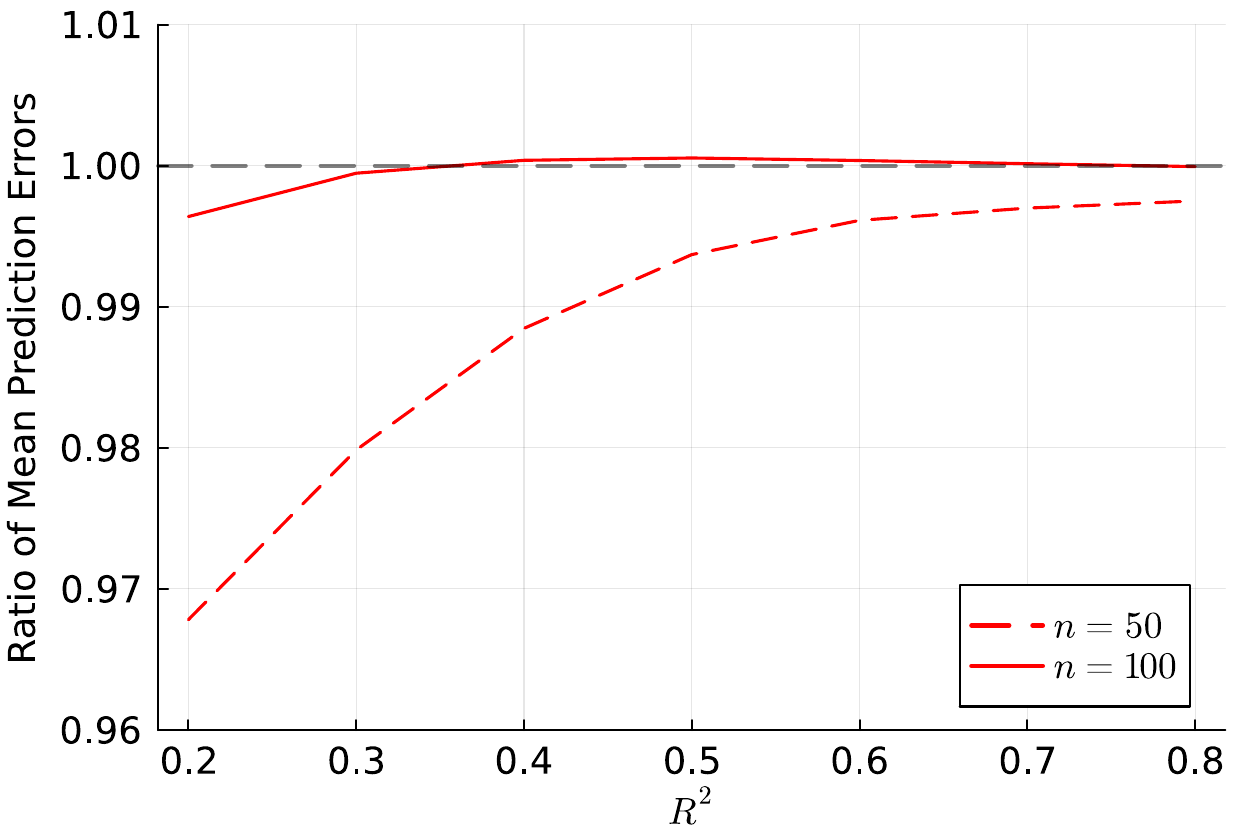}
			\caption{$g$ prior candidates.}
			\label{fig:linear_g}
		\end{subfigure}%
		\begin{subfigure}{0.5\textwidth}
			\centering
			\includegraphics[scale = 0.35]{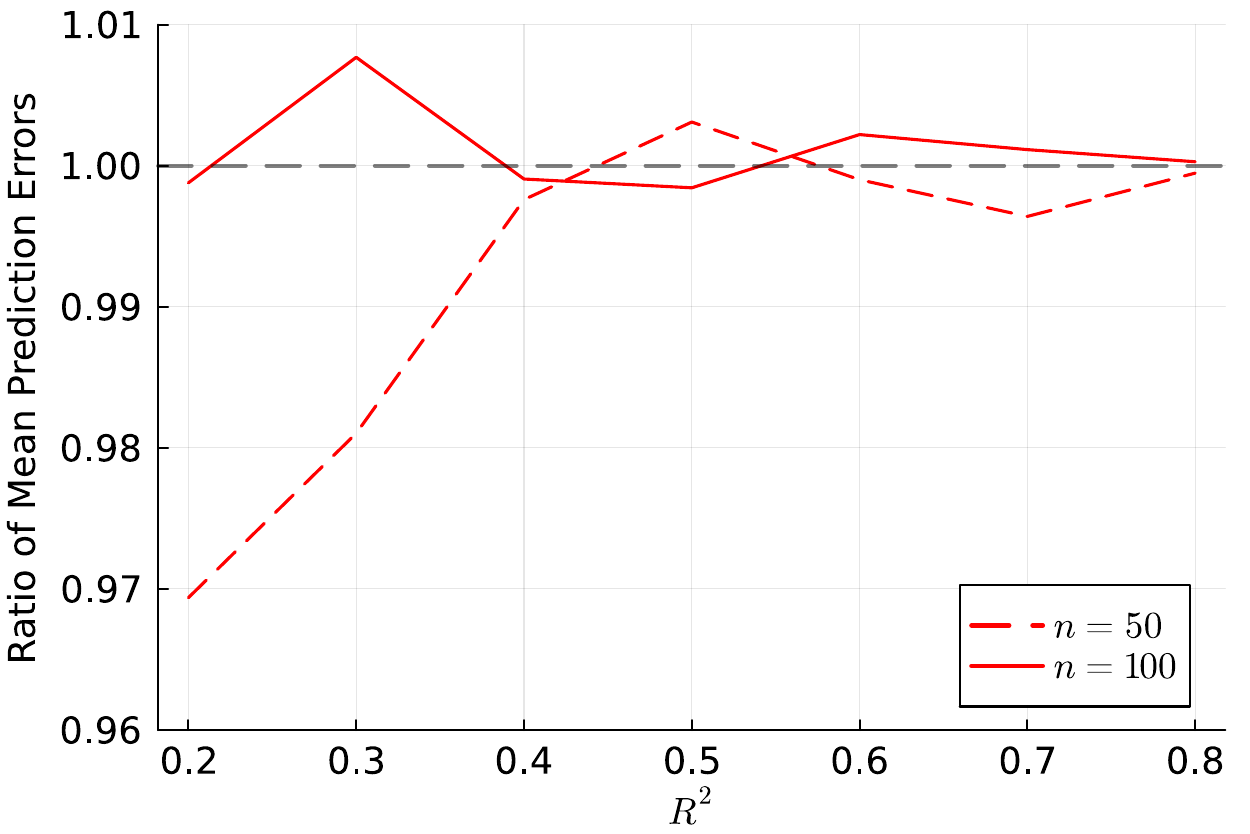}
			\caption{$\mathcal T$ prior candidates.}
			\label{fig:linear_t}
		\end{subfigure}
		
		\caption{Results of simulations made using Zellner's $g$ (left) and the $\mathcal T$ prior (right). A solid line indicates that $n = 50$ observations have been used to fit the candidates, while a dashed line indicates a larger sample size of $n = 100$.}
		\label{fig:linreg}
	\end{figure}
	
	The candidate models fitted all assume a Bernoulli likelihood, and average estimators form compromises between different specifications of some aspect of the assumed candidate priors. As in the linear model simulations, two different sets of candidate are specified.
	
	\begin{itemize}
		\item $\mathcal{M}_\lambda$ collects regression models assuming $\boldsymbol{\beta}_k \sim \mathcal{N}_p \left[\boldsymbol{0}, \; \lambda_k \boldsymbol{I}\right]$, where $\lambda_k$ are equally spaced in $ [10^{-3}, \ 10],$ for $m = 1, \ldots, 50$, and
		
		\item $\mathcal{M}_T$ consists of candidate models where $\boldsymbol{\beta}_k \sim \mathcal{T}_p(\nu_k, \; \lambda \boldsymbol{I}_{p \times p})$, where the degrees of freedom $\nu_k \in \{1, \ldots, 30\}$, where $\lambda = (0.2, 0.2, 0.2, 0.1, 0.1, 0.1, 0.1)$ for the different values of $R^2$.
	\end{itemize}
	
	\noindent Due to the computational intensity of cross-validating Bayesian models, the weights are estimated using 10-fold cross-validation, rather than LOO-CV. Moreover, the point estimates $\hat{\boldsymbol p}_k$ are made using maximum aposteriori estimates rather than posterior means, in order to avoid the computational time required to do posterior sampling in a cross-validation procedure.

	For each $c$, 10 000 replications are performed, and the model average is evaluated against a test set, consisting of 500 observations, as,
	
	\begin{align}\label{eq:ratio_logreg}
		r_p = \frac{\frac{1}{R} \sum_{r = 1}^R (\lVert \boldsymbol{p}  - \bar{\boldsymbol{p}}_{\text{test}} \rVert^2)_r}{\frac{1}{R} \sum_{r = 1}^R (\lVert \boldsymbol{p} - \hat{\boldsymbol p}_{*\text{test}} \rVert^2)_r},
	\end{align}
	
	\noindent where $\hat{\boldsymbol p}_{*\text{test}}$ collects point estimates of validation set probabilities made using the best candidate model in the corresponding set of candidates (i.e. either in $\mathcal{M}_g$ or $\mathcal{M}_T$). The results are displayed in the panels of Figure \ref{fig:logit}.

	\begin{figure}[t!]
		\centering
		\begin{subfigure}{0.5\textwidth}
			\centering
			\includegraphics[scale = 0.35]{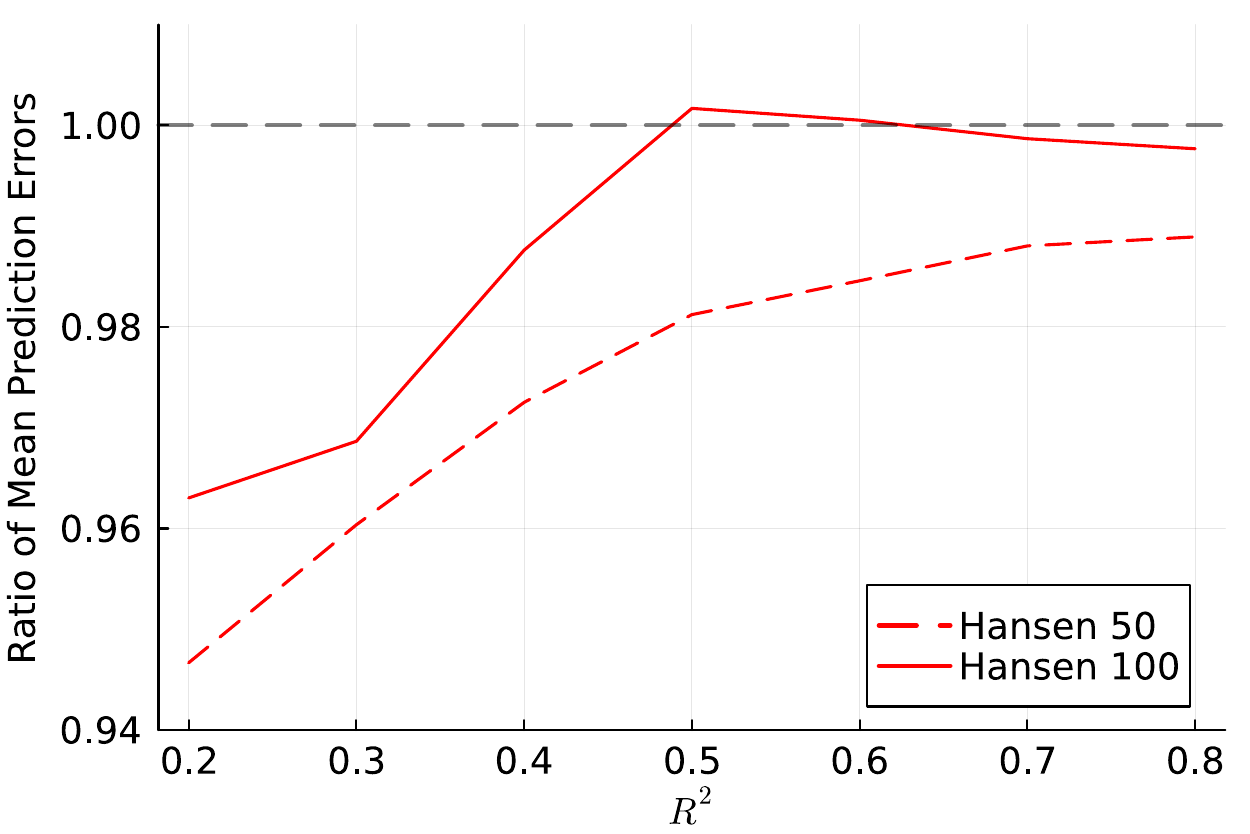}
			\caption{Normal prior candidates}
			\label{fig:logit_g}
		\end{subfigure}%
		\begin{subfigure}{0.5\textwidth}
			\centering
			\includegraphics[scale = 0.35]{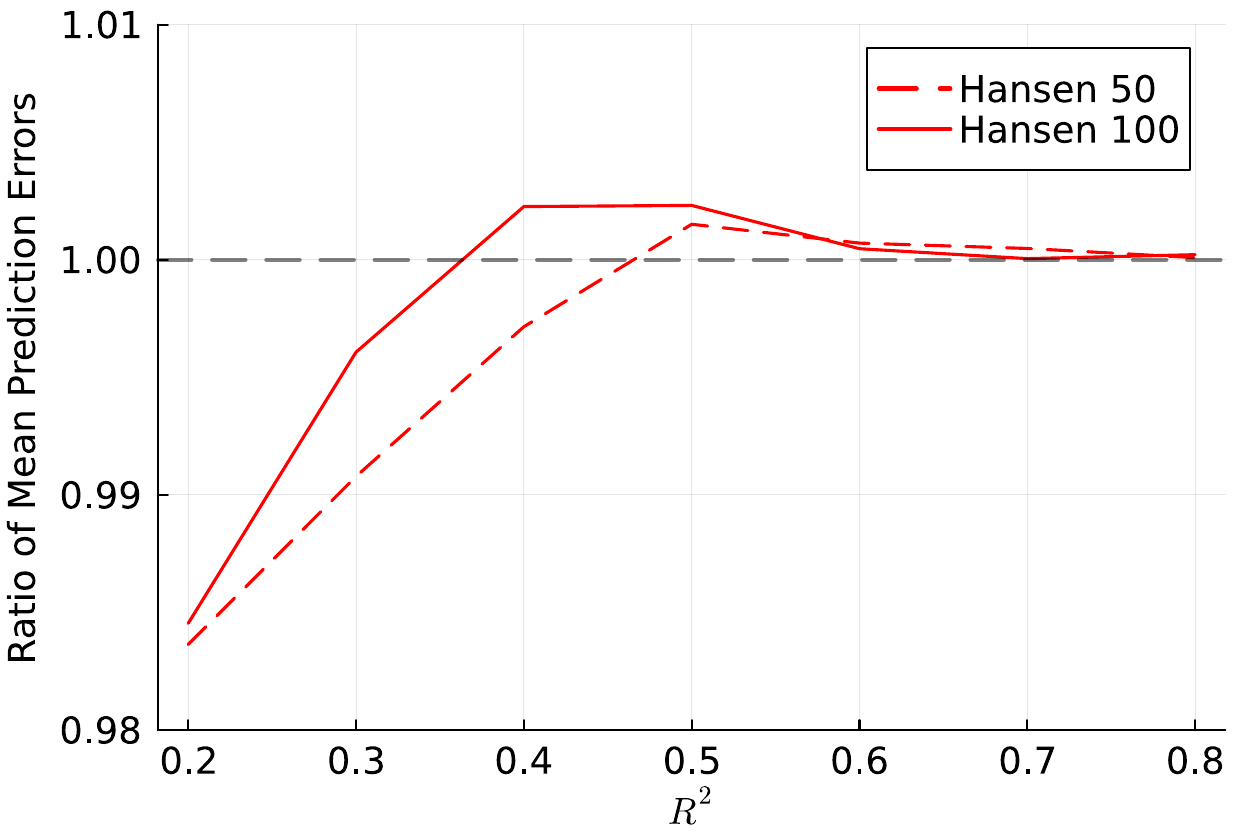}
			\caption{$\mathcal T$ priors candidates}
			\label{fig:logit_t}
		\end{subfigure}
		
		\caption{Results of logistic regression simulations made using normal priors (left) and $\mathcal{T}$ priors (right). A solid line indicates that $n = 50$ observations have been used to fit the candidates, while a dashed line indicates a larger sample size of $n = 100$.}
		\label{fig:logit}
	\end{figure}
	
	\subsubsection{Results}
	
	The results of the simulations studying Bayesian stacking of logistic regression models are similar to those presented in Section \ref{sec:linear_results}. As is illustrated in Figure \ref{fig:logit}, stacking either improves on or performs as well as the best candidate model, and again combining models seems to be especially beneficial when the data are noisy or when the samples are small. Stacking seems to behave similarly in both $\mathcal M_g$ and $\mathcal M_T$.
	
	\section{Conclusions} \label{sec:conclusions}
	
	The paper provides the proof of an oracle property, asserted in Theorem \ref{thm:linear_oracle}, of compromises stacking point predictions made using Bayesian linear regression models, under the prior assumption that regression coefficients are normally distributed. In doing so the paper contributes the first oracle property of its kind for Bayesian stacking, thereby strengthening the case for the use of Bayesian stacking by both Bayesian and frequentist practitioners. In addition, the paper draws from theory on frequentist model averaging, and contributes a discussion of how existing oracle properties may be reinterpreted for application to Bayesian models of a more general form, for example logistic regression or linear regression with non-normal priors. Finally, Bayesian stacking is evaluated in a simulation study, and its performance compared to that of the best candidate model in the stack. The results of the simulation study are intuitive, indicating that stacking is particularly fruitful in small samples and with noisy data. It is reasonable to expect that combining candidate models is more beneficial when data are noisy or sample sizes small, since noise is bound to increase model selection uncertainty. For example, similar tendencies can be observed in both \citet{Zhao2020} and \citet{HansenRacine2012}. 
	
	\subsubsection*{Funding Information}
	Måns Magnusson was funded by The Swedish Research Council Grant 2022-03381.
	
	\newpage
	
	\bibliographystyle{apalike}
	\bibliography{references}

\begin{thebibliography}{}

\bibitem[Akaike, 1973]{Akaike1973}
Akaike, H. (1973).
\newblock {Information Theory and an Extension of the Maximum Likelihood
  Principle}.
\newblock {\em Second International Symposium on Information Theory},
  73:1033--1055.

\bibitem[Ando and Li, 2017]{AndoLi2017}
Ando, T. and Li, K.-C. (2017).
\newblock {A Weight-Relaxed Model Averaging Approach for High-dimensional
  Generalized Linear Models}.
\newblock {\em The Annals of Statistics}, 45(6):2654--2679.

\bibitem[Bernardo and Smith, 1994]{BernardoSmith1994}
Bernardo, J.~M. and Smith, A. F.~M. (1994).
\newblock {\em Bayesian theory}.
\newblock Wiley Series in Probability and Mathematical Statistics. Wiley, New
  York.

\bibitem[Breiman, 1996]{Breiman1996}
Breiman, L. (1996).
\newblock {Heuristics of Instability and Stabilization in Model Selection}.
\newblock {\em The Annals of Statistics}, 24(6):2350--2383.

\bibitem[Charkhi et~al., 2016]{CharkhiClaeskensHansen2016}
Charkhi, A., Claeskens, G., and Hansen, B.~E. (2016).
\newblock {Minimum Mean Squared Error Model Averaging in Likelihood Models}.
\newblock {\em Statistica Sinica}, 26(2):809--.840.

\bibitem[Clyde and Iversen, 2013]{ClydeIversen2013}
Clyde, M.~A. and Iversen, E.~S. (2013).
\newblock {Bayesian Model Averaging in the $\mathcal{M}$-open Framework}.
\newblock In {\em Bayesian Theory and Applications}. Oxford University Press,
  Oxford.

\bibitem[de~Luna et~al., 2011]{deLuna2011}
de~Luna, X., Waernbaum, I., and Richardson, T.~S. (2011).
\newblock {Covariate Selection for the Nonparametric Estimation of an Average
  Treatment Effect}.
\newblock {\em Biometrika}, 98(4):861--875.

\bibitem[Draper, 1995]{Draper1995}
Draper, D. (1995).
\newblock {Assessment and Propagation of Model Uncertainty}.
\newblock {\em Journal of the Royal Statistical Society. Series B,
  Methodological}, 57(1):45--97.

\bibitem[Gao et~al., 2019]{Gao2019}
Gao, Y., Zhang, X., Wang, S., Chong, T. T.-l., and Zou, G. (2019).
\newblock {Frequentist Model Averaging for Threshold Models}.
\newblock {\em Annals of the Institute of Statistical Mathematics},
  71(2):275--306.

\bibitem[Hansen, 2005]{Hansen2005}
Hansen, B.~E. (2005).
\newblock {Challenges for Econometric Model Selection}.
\newblock {\em Econometric Theory}, 21(1):60--68.

\bibitem[Hansen, 2007]{Hansen2007}
Hansen, B.~E. (2007).
\newblock {Least Squares Model Averaging}.
\newblock {\em Econometrica}, 75(4):1175--1189.

\bibitem[Hansen and Racine, 2012]{HansenRacine2012}
Hansen, B.~E. and Racine, J.~S. (2012).
\newblock {Jackknife Model Averaging}.
\newblock {\em Journal of Econometrics}, 167(1):38--46.

\bibitem[Hjort and Claeskens, 2003]{HjortClaeskens2003}
Hjort, N.~L. and Claeskens, G. (2003).
\newblock {Frequentist Model Average Estimators}.
\newblock {\em Journal of The American Statistical Association},
  98(464):879--899.

\bibitem[Hoeting et~al., 1999]{HoetingTutorial}
Hoeting, J.~A., Madigan, D., Raftery, A.~E., and Volinsky, C.~T. (1999).
\newblock {Bayesian Model Averaging: A Tutorial}.
\newblock {\em Statistical Science}, 14(4):382--401.

\bibitem[Kleijn and van~der Vaart, 2012]{KleijnvanderVaart2012}
Kleijn, B. and van~der Vaart, A. (2012).
\newblock {The Bernstein-Von-Mises Theorem Under Misspecification}.
\newblock {\em Electronic Journal of Statistics}, 6:354--381.

\bibitem[Le and Clarke, 2017]{LeClarke2017}
Le, T. and Clarke, B. (2017).
\newblock {A Bayes Interpretation of Stacking for $\mathcal{M}$-Complete and
  $\mathcal{M}$-Open Settings}.
\newblock {\em Bayesian Analysis}, 12(3):807--829.

\bibitem[Mardia et~al., 1979]{Mardia1979}
Mardia, K., Kent, J., and Bibby, J. (1979).
\newblock {\em {Multivariate Analysis}}.
\newblock Probability and Mathematical Statistics : A Series of Monographs and
  Textbooks. Academic Press, London.

\bibitem[Raftery et~al., 1997]{RafteryMadiganHoeting1997}
Raftery, A.~E., Madigan, D., and Hoeting, J.~A. (1997).
\newblock {Bayesian Model Averaging for Linear Regression Models}.
\newblock {\em Journal of The American Statistical Association},
  92(437):179--191.

\bibitem[Ramamoorthi et~al., 2015]{Ramamoorthi2015}
Ramamoorthi, R.~V., Sriram, K., and Martin, R. (2015).
\newblock {On Posterior Concentration in Misspecified Models}.
\newblock {\em Bayesian Analysis}, 10(4):759 -- 789.

\bibitem[Rao et~al., 2001]{Rao2001}
Rao, C.~R., Wu, Y., Konishi, S., and Mukerjee, R. (2001).
\newblock {On Model Selection}.
\newblock {\em Lecture Notes-Monograph Series}, 38:1--64.

\bibitem[Ren et~al., 2023]{Ren2023}
Ren, J., Zhou, F., Li, X., Ma, S., Jiang, Y., and Wu, C. (2023).
\newblock {Robust Bayesian Variable Selection for Gene–environment
  Interactions}.
\newblock {\em Biometrics}, 79(2):684--694.

\bibitem[White, 1982]{White1982}
White, H. (1982).
\newblock {Maximum Likelihood Estimation of Misspecified Models}.
\newblock {\em Econometrica}, 50(1):1--25.

\bibitem[Wolpert, 1992]{Wolpert1992}
Wolpert, D.~H. (1992).
\newblock {Stacked Generalization}.
\newblock {\em Neural Networks}, 5(2):241--259.

\bibitem[Yao et~al., 2018]{Yao2018}
Yao, Y., Vehtari, A., Simpson, D., and Gelman, A. (2018).
\newblock {Using Stacking to Average Bayesian Predictive Distributions (with
  Discussion)}.
\newblock {\em Bayesian Analysis}, 13(3):917--1003.

\bibitem[Zellner, 1986]{Zellner1986}
Zellner, A. (1986).
\newblock {On Assessing Prior Distributions and Bayesian Regression Analysis
  with $g$ Prior Distributions}.
\newblock In Goel, P., Goel, P., De~Finetti, B., and Zellner, A., editors, {\em
  Bayesian Inference and Decision Techniques: Essays in Honor of Bruno de
  Finetti}, Studies in Bayesian Econometrics and Statistics. North-Holland.

\bibitem[Zhang et~al., 2013]{ZhangWanZhou2013}
Zhang, X., Wan, A.~T., and Zou, G. (2013).
\newblock {Model Averaging by Jackknife Criterion in Models with Dependent
  Data}.
\newblock {\em Journal of Econometrics}, 174(2):82--94.

\bibitem[Zhang et~al., 2015]{ZhangZouCarroll2015}
Zhang, X., Zou, G., and Carroll, R.~J. (2015).
\newblock {Model Averaging Based On Kullback-Leibler Distance}.
\newblock {\em Statistica Sinica}, 25(4):1583--1598.

\bibitem[Zhao et~al., 2020]{Zhao2020}
Zhao, S., Liao, J., and Yu, D. (2020).
\newblock {Model Averaging Estimator in Ridge Regression and its Large Sample
  Properties}.
\newblock {\em {Statistical Papers (Berlin, Germany)}}, 61(4):1719--1739.

\bibitem[Zhao et~al., 2019]{ZhaoZhouYang2019}
Zhao, S., Zhou, J., and Yang, G. (2019).
\newblock {Averaging Estimators for Discrete Choice by $M$-fold
  Cross-validation}.
\newblock {\em Economics Letters}, 174:65--69.

\bibitem[Zulj and Jin, 2021]{ZuljJin2021}
Zulj, V. and Jin, S. (2021).
\newblock {Frequentist Model Averaging with Penalization in Generalized Linear
  Models}.
\newblock Unpublished manuscript. Department of Statistics, Uppsala University.

\end{thebibliography}
	
	\newpage
	
	\begin{appendices}
		\titleformat{\section} {\normalfont\Large\bfseries}{Appendix~\thesection}{1em}{}
		
		\section{Statement and proof of Lemma 1} \label{app:eigen_lemma}
		\setcounter{equation}{0}
		\renewcommand{\theequation}{\thesection.\arabic{equation}}
		
		This appendix provides a lemma required for the proof of Theorem \ref{thm:linear_oracle}
		
		\begin{lemma} \label{lemma1}
			Let $\boldsymbol{X}_{n \times p}$ be a matrix of full column rank, $\boldsymbol{S}_k$ a positive definite and symmetric $p \times p$ matrix, and  $\sigma^2 > 0$ a constant. Further, define $$\boldsymbol{P}_k = \boldsymbol{X} (\boldsymbol{X}^T \boldsymbol{X} + \sigma^2 \boldsymbol{S}_k^{-1})^{-1} \boldsymbol{X}^T.$$ Then $\lambda_{\text{max}}(\boldsymbol{P}_k) \leq 1$ for every $m \in \left\{1, \ldots, M \right\}$. Moreover, for every $\boldsymbol{w} \in \mathcal{W}$, $\boldsymbol{P}(\boldsymbol{w}) = \sum_k w_k \boldsymbol{P}_k$ is positive semi-definite with $\lambda_{\text{max}}\{\boldsymbol{P}(\boldsymbol{w})\} \leq 1$.
		\end{lemma}
		
		\begin{proof}
			Suppose $\boldsymbol{X}$ has singular value decomposition $\boldsymbol{X} = \boldsymbol{UDV}^T$ where $\boldsymbol{U}$ is $n \times n$, $\boldsymbol{V}$ is $p \times p$, and both are orthogonal. Moreover, since $\boldsymbol{X}$ has column rank $p$, $\boldsymbol{D}$ is $n \times p$, defined as
			
			\begin{align*}
				\boldsymbol{D} =
				\begin{pmatrix}
					\boldsymbol{D}_0  \\
					\boldsymbol{0} 
				\end{pmatrix},
			\end{align*}
			
			\noindent where $\boldsymbol{D}_0$ is a $p \times p$ diagonal matrix with positive diagonal values, and the $\boldsymbol{0}$ matrix is $(n - p) \times p$. It follows that
			
			\begin{align*}
				\left(\boldsymbol{X}^T \boldsymbol{X} + \sigma^2 \boldsymbol{S}_k^{-1} \right)^{-1} 
				& = \left(\boldsymbol{V D}^T\boldsymbol{U}^T  \boldsymbol{UD V}^T + \sigma^2 \boldsymbol{S}_k^{-1} \right)^{-1} \\
				& = \left(\boldsymbol{V D}^T \boldsymbol{D V}^T + \sigma^2 \boldsymbol{S}_k^{-1} \right)^{-1} \\
				& = \left\{\boldsymbol{V}\left(\boldsymbol{D}^T \boldsymbol{D} + \sigma^2 \boldsymbol{V}^T\boldsymbol{S}_k^{-1} \boldsymbol{V} \right) \boldsymbol{V}^T\right\}^{-1} \\
				& = \left(\boldsymbol{V}^T\right)^{-1} \left(\boldsymbol{D}_0^2 + \sigma^2 \boldsymbol{V}^T\boldsymbol{S}_k^{-1} \boldsymbol{V} \right)^{-1} \boldsymbol{V}^{-1} \\
				& = \boldsymbol{V} \left(\boldsymbol{D}_0^2 + \sigma^2 \boldsymbol{V}^T\boldsymbol{S}_k^{-1} \boldsymbol{V} \right)^{-1} \boldsymbol{V}^T,
			\end{align*}
			
			\noindent and thus that
			
			\begin{align*}
				\boldsymbol{P}_k & = \boldsymbol{X} \left(\boldsymbol{X}^T \boldsymbol{X} + \sigma^2 \boldsymbol{S}_k^{-1} \right)^{-1} \boldsymbol{X}^T \\
				& = \boldsymbol{UD V}^T  \boldsymbol{V} \left(\boldsymbol{D}_0^2 + \sigma^2 \boldsymbol{V}^T\boldsymbol{S}_k^{-1} \boldsymbol{V} \right)^{-1} \boldsymbol{V}^T \boldsymbol{V D}^T\boldsymbol{U}^T  \\
				& = \boldsymbol{UD} \left(\boldsymbol{D}_0^2 + \sigma^2 \boldsymbol{V}^T\boldsymbol{S}_k^{-1} \boldsymbol{V} \right)^{-1} \boldsymbol{D}^T\boldsymbol{U}^T.
			\end{align*}
			
			Let $\lVert \boldsymbol{A} \rVert_{2} = \lambda^{1/2}_{\text{max}}(\boldsymbol{A}^T \boldsymbol{A})$ denote the spectral norm of the matrix $\boldsymbol{A}$. Also, note that for symmetric $\boldsymbol{A}$, $\lVert \boldsymbol{A} \rVert_{2} = \lambda_{\text{max}}(\boldsymbol{A})$.  Then, using the above decomposition of $\boldsymbol{P}_k$,
			
			\begin{align*}
				\lVert \boldsymbol{P}_k \rVert_{2} 
				& = \Big \lVert \boldsymbol{UD} \left(\boldsymbol{D}_0^2 + \sigma^2 \boldsymbol{V}^T\boldsymbol{S}_k^{-1} \boldsymbol{V} \right)^{-1} \boldsymbol{D}^T\boldsymbol{U}^T \Big \rVert_{2} \\
				& = \lambda_{\text{max}} \left\{ \boldsymbol{UD} \left(\boldsymbol{D}_0^2 + \sigma^2 \boldsymbol{V}^T\boldsymbol{S}_k^{-1} \boldsymbol{V} \right)^{-1} \boldsymbol{D}^T\boldsymbol{U}^T \right\} \\
				& =  \lambda_{\text{max}} \left\{  \boldsymbol{D}^T\boldsymbol{U}^T \boldsymbol{UD} \left(\boldsymbol{D}_0^2 + \sigma^2 \boldsymbol{V}^T\boldsymbol{S}_k^{-1} \boldsymbol{V} \right)^{-1} \right\} \\
				& =  \lambda_{\text{max}} \left\{  \boldsymbol{D}_0^2 \left(\boldsymbol{D}_0^2 + \sigma^2 \boldsymbol{V}^T\boldsymbol{S}_k^{-1} \boldsymbol{V} \right)^{-1} \right\} \\
				& =  \lambda_{\text{max}} \left\{  \boldsymbol{D}_0^2 \left[ (\boldsymbol I + \sigma^2 \boldsymbol{V}^T\boldsymbol{S}_k^{-1} \boldsymbol{V}\boldsymbol{D}_0^{-2}) \boldsymbol{D}_0^2   \right]^{-1} \right\} \\
				& = \lambda_{\text{max}} \left\{ \left[\boldsymbol{I}+ \sigma^2 \boldsymbol{V}^T\boldsymbol{S}_k^{-1} \boldsymbol{V} \boldsymbol{D}_0^{-2} \right]^{-1} \right\} \\
				& = \lambda_{\text{min}}^{-1} \left\{ \boldsymbol{I}+ \sigma^2 \boldsymbol{V}^T\boldsymbol{S}_k^{-1} \boldsymbol{V} \boldsymbol{D}_0^{-2}  \right\} \\
				& = \bigg[1 + \sigma^2 \lambda_{\text{min}} \left\{ \boldsymbol{V}^T\boldsymbol{S}_k^{-1} \boldsymbol{V} \boldsymbol{D}_0^{-2}  \right\}\bigg]^{-1} \\
				& \leq 1,
			\end{align*}
			
			\noindent where the inequality follows from the argument given in the ensuing paragraph.
			
			During the remainder of this proof, let $\boldsymbol{A} = \boldsymbol{V}^T\boldsymbol{S}_k^{-1} \boldsymbol{V}$ and $\boldsymbol{B} = \boldsymbol{D}_0^{-2}$. Since $\boldsymbol{A}$ is symmetric, it admits the decomposition $\boldsymbol{A} = \boldsymbol{A}^{1/2} \boldsymbol{A}^{1/2}$, where $\boldsymbol{A}^{1/2}$ is also symmetric. Moreover, $\boldsymbol{B}$ is $p \times p$, positive definite, and diagonal with $\boldsymbol{B}^{-1} > 0$. Thus, $\boldsymbol{B}^{-1}$ admits the similar decomposition $\boldsymbol{B}^{-1} = \boldsymbol{B}^{-1/2} \boldsymbol{B}^{-1/2}$. As such,
			
			\begin{align*}
				\lambda_{\text{min}} \left\{ \boldsymbol{V}^T\boldsymbol{S}_k^{-1} \boldsymbol{V} \boldsymbol{D}_0^{-2}  \right\} 
				& = \lambda_{\text{min}} \left\{ \boldsymbol{A} \boldsymbol{B}^{-1}  \right\} 
				= \lambda_{\text{min}} \left\{ \boldsymbol{A} \boldsymbol{B}^{-1/2} \boldsymbol{B}^{-1/2}  \right\} \\
				& = \lambda_{\text{min}} \left\{ \boldsymbol{B}^{-1/2} \boldsymbol{A} \boldsymbol{B}^{-1/2}  \right\} \\
				& = \lambda_{\text{min}} \left\{ \boldsymbol{B}^{-1/2} \boldsymbol{A}^{1/2} \boldsymbol{A}^{1/2} \boldsymbol{B}^{-1/2}  \right\} \\
				& = \lambda_{\text{min}} \left\{ \left(\boldsymbol{A}^{1/2} \boldsymbol{B}^{-1/2}\right)^T  \boldsymbol{A}^{1/2} \boldsymbol{B}^{-1/2} \right\} \\
				& \geq 0.
			\end{align*}
			
			\noindent Lastly, $\sigma^2 > 0$, so $\sigma^2 \lambda_{\text{min}} \left\{ \boldsymbol{V}^T\boldsymbol{S}_k^{-1} \boldsymbol{V} \boldsymbol{D}_0^{-2} \right\} \geq 0$, and
			
			\begin{align*}
				0 \leq \bigg[1 + \sigma^2 \lambda_{\text{min}} \left\{ \boldsymbol{V}^T\boldsymbol{S}_k^{-1} \boldsymbol{V} \boldsymbol{D}_0^{-2}  \right\}\bigg]^{-1} 
				\leq 1.
			\end{align*}
			
			Now, let $\lambda_j(\boldsymbol A)$ denote the $j$:th largest eigenvalue of some matrix $\boldsymbol A$. Then, by the Weyl inequalities, it follows that
			
			\begin{align*}
				\lambda_{i+j-1} \left( \sum_{m = 1}^M w_k \boldsymbol P_k \right)  
				\leq \lambda_i(w_1 \boldsymbol P_1) + \lambda_{j} \left( \sum_{m = 2}^M w_k \boldsymbol P_k \right)  
				= w_1 \lambda_i (\boldsymbol P_1)  + \lambda_{j} \left( \sum_{m = 2}^M w_k \boldsymbol P_k \right).
			\end{align*}
			
			\noindent Repeated application of the same inequality gives
			
			\begin{align*}
				\lambda_{j} \left( \sum_{m = 2}^M w_k \boldsymbol P_k \right)  
				\leq \lambda_1(w_2 \boldsymbol P_2) + \lambda_{j} \left( \sum_{m = 3}^M w_k \boldsymbol P_k \right)  
				= w_2 \lambda_1 (\boldsymbol P_2)  + \lambda_{j} \left( \sum_{m = 3}^M w_k \boldsymbol P_k \right),
			\end{align*}
			
			\noindent and
			
			\begin{align*}
				\lambda_{j} \left( \sum_{m = 3}^M w_k \boldsymbol P_k \right)  
				\leq \lambda_1(w_3 \boldsymbol P_3) + \lambda_{j} \left( \sum_{m = 4}^M w_k \boldsymbol P_k \right)  
				= w_3 \lambda_1 (\boldsymbol P_3)  + \lambda_{j} \left( \sum_{m = 4}^M w_k \boldsymbol P_k \right),
			\end{align*}
			
			\noindent repeating in the same manner until
			
			\begin{eqnarray*}
				& & \lambda_{j} \left( w_{M-1} \boldsymbol P_{M-1} +  w_{M} \boldsymbol P_{M}  \right)  
				\leq \lambda_1(w_{M-1} \boldsymbol P_{M-1}) + \lambda_{j} \left( w_k \boldsymbol P_k \right)  \\
				&=& w_{M-1} \lambda_1 (\boldsymbol P_{M-1})  + \lambda_{j} \left(w_k \boldsymbol P_k \right).
			\end{eqnarray*}
			
			\noindent Thus,
			
			\begin{align*}
				\lambda_{i+j-1} \left( \sum_{m = 1}^M w_k \boldsymbol P_k \right) 
				\leq w_1 \lambda_i (\boldsymbol P_1) + w_{M} \lambda_j (\boldsymbol P_{M}) + \sum_{m = 2}^{M-1} w_k \lambda_1(\boldsymbol P_k).
			\end{align*}
			
			Finally, let $\lambda^* = \max_k \{\lambda_{\text{max}} (\boldsymbol P_k) \}$. Then, 
			
			\begin{align*}
				\lambda_{i+j-1} \left( \sum_{m = 1}^M w_k \boldsymbol P_k \right) 
				\leq w_1 \lambda^*+ w_{M-1} \lambda^* + \sum_{m = 2}^{M-1} w_k \lambda^*
				= \lambda^* \sum_{m = 1}^{M} w_k = \lambda^* = 1.
			\end{align*}
			
			The Weyl inequalities can be used similarly to show that the eigenvalues of $\sum_k w_k \boldsymbol P_k$ are bounded below by 0, completing the proof of the lemma.
		\end{proof}
		
		\section{Proof of Theorem 1} \label{app:lin_oracle}
		
		\setcounter{theorem}{0}
		\setcounter{equation}{0}
		\renewcommand{\theequation}{\thesection.\arabic{equation}}
		
	 	\textit{Note: The proof that follows is based on the proof of Theorem 2 in \cite{Zhao2020} to a large extent, with some generalizing additions. As such, the proof given below is a sketch of the proof provided by \citet{Zhao2020}, where the alterations made are entered chronologically.}
	
		\vspace{0.5cm}
		
		\noindent Let $\boldsymbol{P}_k = \boldsymbol{X} \left(\boldsymbol{X}^T \boldsymbol{X} + \sigma^2 \boldsymbol{S}_k^{-1} \right)^{-1} \boldsymbol{X}^T$,  $\boldsymbol{A}_k = \boldsymbol{I}_{n \times n} - \boldsymbol{P}_k$, and let $\boldsymbol{Q}_k$ be the $n \times n$ diagonal matrix with elements $Q^K_{ii} = P^K_{ii} /(1 - P^K_{ii})$, where $P^K_{ii}$ are diagonal elements of $\boldsymbol{P}_k$. Moreover, let $\boldsymbol{B}_k = \boldsymbol{Q}_k \boldsymbol{A}_k$, and define the weighted averages
		
		\begin{align*}
			\boldsymbol{A}(\boldsymbol{w}) = \sum_{k= 1}^{K} w_k \boldsymbol{A}_k \ \text{ and } \ \boldsymbol{B}(\boldsymbol{w}) = \sum_{k= 1}^{K} w_k \boldsymbol{B}_k.
		\end{align*}
		
		\noindent Considering Equation (6) of \cite{Zhao2020}, the equality extends without issue to the case where $k_k \boldsymbol{I}_{n \times n}$ is replaced by the positive semidefinite matrix $\boldsymbol{S}_k$. To see this, note that $\boldsymbol{X}^T\boldsymbol{X} = \sum_{i = 1}^n \boldsymbol{x}_i \boldsymbol{x}_i^T$, and that $\boldsymbol{X}_{-j} \boldsymbol{X}_{-j}^T = \sum_{i \neq j} \boldsymbol{x}_i \boldsymbol{x}_i^T$. It follows that $\boldsymbol{X}_{-j} \boldsymbol{X}_{-j}^T = \boldsymbol{X}^T\boldsymbol{X} - \boldsymbol{x}_j \boldsymbol{x}_j^T$, and that
		
		\begin{align*}
			& \left( \boldsymbol{X}_{-i}  \boldsymbol{X}_{-i}^T + \sigma^2 \boldsymbol{S}_k^{-1} \right)^{-1}  \\
			& \hspace{0.25cm} =  \left( \boldsymbol{X}^T\boldsymbol{X} + \sigma^2 \boldsymbol{S}_k^{-1} + (- \boldsymbol{x}_i) \boldsymbol{x}_i^T \right)^{-1} \\
			& \hspace{0.25cm}  = \left(\boldsymbol{X}^T\boldsymbol{X} + \sigma^2 \boldsymbol{S}_k^{-1} \right)^{-1} + 
			\frac{\left(\boldsymbol{X}^T\boldsymbol{X} + \sigma^2 \boldsymbol{S}_k^{-1} \right)^{-1}\boldsymbol{x}_i \boldsymbol{x}_i^T \left(\boldsymbol{X}^T\boldsymbol{X} + \sigma^2 \boldsymbol{S}_k^{-1} \right)^{-1}}{1 - \boldsymbol{x}_i^T \left(\boldsymbol{X}^T\boldsymbol{X} + \sigma^2 \boldsymbol{S}_k^{-1} \right)^{-1} \boldsymbol{x}_i},
		\end{align*}
		
		\noindent by application of property (A.2.4f) in \cite{Mardia1979}. Thus, the decomposition
		
		\begin{align*}
			\boldsymbol{y} - \sum_{k= 1}^{K} w_k \tilde{\boldsymbol{\mu}}_k = \left[\boldsymbol{B}(\boldsymbol{w}) + \boldsymbol{A}(\boldsymbol{w}) \right] \boldsymbol{y}
		\end{align*}
		
		\noindent follows.
		
		Set $\boldsymbol{M}(\boldsymbol{W}) = \boldsymbol{A}(\boldsymbol{w})\boldsymbol{B}(\boldsymbol{w}) + \boldsymbol{A}^T(\boldsymbol{w})\boldsymbol{B}(\boldsymbol{w}) + \boldsymbol{B}^T(\boldsymbol{w})\boldsymbol{B}(\boldsymbol{w})$. It follows, as in \cite{Zhao2020}, that the cross validation criterion
		
		\begin{align*}
			\left\lVert \boldsymbol{y} - \sum_{k= 1}^{K} w_k \tilde{\boldsymbol{\mu}} \right\rVert ^2
			= L(\boldsymbol{w}) + r(\boldsymbol{w}) + \boldsymbol{\varepsilon}^T\boldsymbol{\varepsilon},
		\end{align*}
		
		\noindent where
	
		\begin{align*}
			r(\boldsymbol{w})  = 2 \boldsymbol{\mu} \boldsymbol{A}(\boldsymbol{w}) \boldsymbol{\varepsilon} - 2 \boldsymbol{\varepsilon}^T \boldsymbol{P}(\boldsymbol{w}) \boldsymbol{\varepsilon} 
			+ \boldsymbol{\mu}^T \boldsymbol{M}(\boldsymbol{w}) \boldsymbol{\mu} + 2\boldsymbol{\mu}^T \boldsymbol{M}(\boldsymbol{w}) \boldsymbol{\varepsilon} + \boldsymbol{\varepsilon}^T \boldsymbol{M}(\boldsymbol{w}) \boldsymbol{\varepsilon},
		\end{align*}
		
		\noindent with $\boldsymbol{P}(\boldsymbol{w}) = \sum_k w_k \boldsymbol{P}_k$. Since $\boldsymbol{w}$ is not a function of $\boldsymbol{\varepsilon}^T \boldsymbol{\varepsilon}$, \citet{Zhao2020} assert, based on Lemma 1 in \cite{Gao2019}, that showing
		
		\begin{align} \label{eq:oracle_cond1}
			\sup_{\boldsymbol{w} \in \mathcal{W}} \frac{r(\boldsymbol{w})}{R(\boldsymbol{w})} = o_p(1),
		\end{align}
		
		\noindent and
		
		\begin{align} \label{eq:oracle_cond2}
			\sup_{\boldsymbol{w} \in \mathcal{W}} \left \lvert \frac{L(\boldsymbol{w})}{R(\boldsymbol{w})} -1 \right  \rvert = o_p(1)
		\end{align}
		
		\noindent is sufficient to prove Theorem 3.1.
		
		Concerning condition \eqref{eq:oracle_cond1}, \cite{Zhao2020} show that, under the additional assumption that $\lambda_{\text{max}}(\boldsymbol{\Omega}) = \bar{\sigma}^2 < \infty$ a.s., 
		
		\begin{itemize}
			\item $\sup_{\boldsymbol{w} \in \mathcal{W}} R^{-1}(\boldsymbol{w}) |\boldsymbol{\mu} \boldsymbol{A}(\boldsymbol{w}) \boldsymbol{\varepsilon}| = o_p(1)$,
			
			\item $\sup_{\boldsymbol{w} \in \mathcal{W}} R^{-1}(\boldsymbol{w})  \lvert \boldsymbol{\varepsilon}^T \boldsymbol{P}(\boldsymbol{w}) \boldsymbol{\varepsilon} - \tr{\boldsymbol{P}(\boldsymbol{w}) \boldsymbol{\Omega}} \rvert = o_p(1)$,
			
			\item $\sup_{\boldsymbol{w} \in \mathcal{W}} R^{-1}(\boldsymbol{w})  \lvert \tr{\boldsymbol{P}(\boldsymbol{w}) \boldsymbol{\Omega}} \rvert = o(1)$, a.s., 
			
			\item $\sup_{\boldsymbol{w} \in \mathcal{W}} R^{-1}(\boldsymbol{w})  \lvert \boldsymbol{\mu}^T \boldsymbol{M}(\boldsymbol{w}) \boldsymbol{\mu} \rvert = o(1)$, a.s., 
			
			\item $\sup_{\boldsymbol{w} \in \mathcal{W}} R^{-1}(\boldsymbol{w})  \lvert \boldsymbol{\mu}^T \boldsymbol{M}(\boldsymbol{w}) \boldsymbol{\varepsilon} \rvert = o_p(1)$,
			
			\item $\sup_{\boldsymbol{w} \in \mathcal{W}} R^{-1}(\boldsymbol{w})  \lvert \boldsymbol{\varepsilon}^T \boldsymbol{M}(\boldsymbol{w}) \boldsymbol{\varepsilon} \rvert = o_p(1)$,
		\end{itemize}
		
		\noindent from which \eqref{eq:oracle_cond1} follows. The same procedure is valid for the current proof.
		
		Regarding \eqref{eq:oracle_cond2}, note that
	
		\begin{align*}
			R(\boldsymbol{w}) & = \mathbb{E} \left[ \lVert \boldsymbol{\mu} - \bar{\boldsymbol{\mu}} \rVert^2 \Big \vert  \boldsymbol{Z}\right]
			= \mathbb{E} \left[ \Big\lVert \boldsymbol{\mu} - \boldsymbol{P}(\boldsymbol{w}) \boldsymbol{y} \Big\rVert^2 \ \Big \vert  \boldsymbol{Z}\right] \\
			& = \mathbb{E} \left\{\Big\lVert \left[\boldsymbol{\mu} - \boldsymbol{P}(\boldsymbol{w}) \boldsymbol{\mu}\right]  - \left[\boldsymbol{P}(\boldsymbol{w}) (\boldsymbol{y} - \boldsymbol{\mu}) \right] \Big\rVert^2 \ \Big \vert  \boldsymbol{Z}\right\} \\
			& = \mathbb{E} \left[ \Big \rVert \boldsymbol{A}(\boldsymbol{w})\boldsymbol{\mu}\Big \rVert^2 \ \Big \vert  \boldsymbol{Z} \right] 
			- 2\mathbb{E} \left[ \boldsymbol{\mu}^T\boldsymbol{A}(\boldsymbol{w}) \boldsymbol{P}(\boldsymbol{w})\boldsymbol{\varepsilon}\ \Big \vert  \boldsymbol{Z} \right]
			+ \mathbb{E} \left[\Big \rVert \boldsymbol{P}(\boldsymbol{w})\boldsymbol{\varepsilon}\Big \rVert^2 \ \Big \vert \boldsymbol{Z} \right] \\ 
			& = \Big \rVert \boldsymbol{A}(\boldsymbol{w})\boldsymbol{\mu}\Big \rVert^2 + \tr{\boldsymbol{P}(\boldsymbol{w}) \boldsymbol{\Omega} \boldsymbol{P}(\boldsymbol{w})},
		\end{align*}
		
		\noindent and, that,
		
		\begin{align*}
			\lvert L(\boldsymbol{w}) - R(\boldsymbol{w}) \rvert 
			= \big \rvert \big \rVert \boldsymbol{P}(\boldsymbol{w})\boldsymbol{\varepsilon}\big \rVert^2 - \tr{\boldsymbol{P}(\boldsymbol{w}) \boldsymbol{\Omega} \boldsymbol{P}(\boldsymbol{w})} - 2\boldsymbol{\mu}^T\boldsymbol{A}(\boldsymbol{w}) \boldsymbol{P}(\boldsymbol{w})\boldsymbol{\varepsilon} \big \lvert.
		\end{align*}
		
		\noindent Thus, showing \eqref{eq:oracle_cond2} reduces to proving
		
		\begin{itemize}
			\item $\sup_{\boldsymbol{w} \in \mathcal{W}} R^{-1}(\boldsymbol{w}) \big \rVert \boldsymbol{P}(\boldsymbol{w})\boldsymbol{\varepsilon}\big \rVert^2 = o_p(1)$,
			
			\item $\sup_{\boldsymbol{w} \in \mathcal{W}} R^{-1}(\boldsymbol{w}) \tr{\boldsymbol{P}(\boldsymbol{w}) \boldsymbol{\Omega} \boldsymbol{P}(\boldsymbol{w})} = o(1)$, a.s., and
			
			\item $\sup_{\boldsymbol{w} \in \mathcal{W}} R^{-1}(\boldsymbol{w}) \boldsymbol{\mu}^T\boldsymbol{A}(\boldsymbol{w}) \boldsymbol{P}(\boldsymbol{w})\boldsymbol{\varepsilon} = o_p(1)$.
		\end{itemize}
		
		\noindent Using Lemma \ref{lemma1}, the three statements are proven as in \cite{Zhao2020}. This completes the proof.

		\section{Posterior Concentration} \label{app:concentration}
		
		\setcounter{theorem}{0}
		\setcounter{equation}{0}
		\renewcommand{\theequation}{\thesection.\arabic{equation}}
		
		For ease of translation, this appendix follows the notational framework set up by \citet{Ramamoorthi2015}. Now, suppose the data generating process is unknown. That is, suppose there is a true density function $f_{0\boldsymbol{x}}$ such that $Y | \boldsymbol{x} \sim f_{0\boldsymbol{x}}$, and let $\mathbb{E}_{\boldsymbol{x}}$ denote the expectation operator with respect to the density $f_{0\boldsymbol{x}}$. Moreover, let $\mathcal{F} = \{f_t : t \in [-K, K], \ K \in \mathbb{R}^+\}$ be a family of densities, $\Theta$ be a class of continuous functions mapping from $\mathcal{X}$ to $[-K, K]$, and assume that the specified models are of the form $Y_i \sim f_{\theta_{\boldsymbol{x}_i}}$, where $\theta_{\boldsymbol{x}_i} = \theta(\boldsymbol{x}_i)$.

		Let $\Pi(\cdot)$ be a prior on the parameter space $\Theta$, and $\Pi(A | \boldsymbol{y}, \boldsymbol{X})$ denote the posterior probability of a set $A$ given the outcomes $\boldsymbol{y}$ and covariates $\boldsymbol{X}$. Moreover, let $d(\theta_1, \theta_2) = \sup_{\boldsymbol{x} \in \mathcal{X}} |\theta_1(\boldsymbol{x}) - \theta_2(\boldsymbol{x})|$, $\theta^*$ be the minimizer of $\mathbb{E}_{\boldsymbol x}  [\log \left(f_{0 \boldsymbol x} / f_{\theta_x}\right)]$, and $U^c = \left\{\theta \in \Theta : d(\theta, \theta^*) > \epsilon \right\}$. Then, \citet{Ramamoorthi2015} proves that $\Pi(U^c | \boldsymbol{y}, \boldsymbol{X}) \rightarrow 0$, almost surely in $P_0$, where, $P_0$ is the infinite product measure $f_{0 \boldsymbol x_1} \times f_{0 \boldsymbol x_2} \times \cdots$. The proof is valid under Assumptions \ref{ass:Rama1} -- \ref{ass:Rama5}.

		\begin{enumerate}[label = AB.\arabic*, font={\bfseries}, leftmargin = 2 cm]
			
			\item The covariate space $\mathcal{X}$ is compact with respect to a norm $\lVert \cdot \rVert$ and $\Theta$ is a compact subset of continuous functions from $\mathcal{X} \rightarrow \mathbb{R}$  endowed with the sup-norm metric $d(\theta_1, \theta_2) = \sup_{\boldsymbol{x} \in \mathcal{X}} |\theta_1(\boldsymbol{x}) - \theta_2(\boldsymbol{x})|$. \label{ass:Rama1}
			
			\item For any given $\boldsymbol{x}_0 \in \mathcal{X}$ and $\delta' > 0$, let $A_{\boldsymbol{x}_0, \delta'} = \{\boldsymbol{x} : \lVert \boldsymbol{x} - \boldsymbol{x}_0 \rVert < \delta'\}$ and $I_{A_{\boldsymbol{x}_0, \delta'}}(\boldsymbol{x})$ be the indicator which is 1 when $\boldsymbol{x}_0 \in A_{\boldsymbol{x}_0, \delta'}$ and 0 otherwise. Then, $\kappa (\boldsymbol{x}_0, \delta') = \operatorname{liminf}_{n \geq 1} \frac{1}{n} \sum_{i = 1}^{n}I_{A_{\boldsymbol{x}_0, \delta'}}(\boldsymbol{x}) > 0$. \label{ass:Rama2}
			
			\item  $\exists \ \theta^* \in \Theta$ such that $\theta^*_{\boldsymbol{x}} = \operatorname{argmin}_{t \in [-K, K]} \mathbb{E}_{\boldsymbol{x}} \left[ \log \left( \frac{f_{0 \boldsymbol{x}}}{f_t}\right)\right], \forall \boldsymbol{x} \in \mathcal{X}$ and $\theta^*$ is in the sup-norm support of $\Pi$. \label{ass:Rama3}
			
			\item $\mathbb{E}_{\boldsymbol{x}} \left[ \log \left( \frac{f_t}{f_{t'}} \right) \right]$ and, for every $\alpha \in (0, 1)$, $\mathbb{E}_{\boldsymbol{x}} \left[  \left( \frac{f_t}{f_{t'}} \right)^{\alpha} \right]$ are continuous function in $(\boldsymbol{x}, t, t' \in \mathcal{X} \times [-K, K]^2$ and $\mathbb{E}_{\boldsymbol{x}} \left[ \log^2 \left( \frac{f_t}{f_{t'}} \right) \right]$ is uniformly bounded for $(\boldsymbol{x}, t, t') \in \mathcal{X} \times [-K, K]^2$. \label{ass:Rama4}
			
			\item For any $\epsilon > 0, \ \exists  \ \delta \in (0, 1)$ such that 
			
			\begin{align*}
				\left\{t \in [-K, K]: \mathbb{E}_{\boldsymbol{x}} \left[ \log \frac{f_{\theta^*_{\boldsymbol{x}}}}{f_t} \right] < \delta  \right\} \subseteq \{t : |t - \theta_{\boldsymbol{x}}^*| < \epsilon \}, \forall \boldsymbol{x} \in \mathcal{X}.
			\end{align*} \label{ass:Rama5}
		\end{enumerate}
		
		\subsection*{Logistic regression}
		
		In the current paper, $Y_i$ is modeled using finite dimensional logistic regression, meaning that 
		
		\begin{align*}
			\mathcal{F} = \left\{ [\operatorname{expit}(t)]^y [1 - \operatorname{expit}(t)]^{1-y}  : t \in [-K, K] \right\},
		\end{align*}
		
		\noindent where $\operatorname{expit}(t) = (1 + e^{-t})^{-1}$ and $\theta_{\boldsymbol {x}} = \boldsymbol {x}^T \boldsymbol {\beta}$. Clearly, this set-up represents a case of model misspecification, since $f_{0\boldsymbol {x}}$ is assumed not to belong to $\mathcal{F}$. 
		
		According to \citet{Ramamoorthi2015}, an example of $\Theta$ that satisfies Assumption \ref{ass:Rama1} is any class of smooth functions defined on a compact set $\mathcal{X}$ and parameterized by finitely many parameters that also take values on some compact set. In such cases, the sup-norm is also equivalent to the Euclidean norm. Here, all candidates depend on finitely many parameters and, in practice, they realistically belong to some compact set. Moreover, the set $\mathcal{X}$ can often be closed and bounded and thus also compact, without much limitation. Assumption \ref{ass:Rama2} concerns only the covariates, and as such does not influence whether the theorem can be applied to logistic regression models or not.
		
		For Assumption \ref{ass:Rama3}, $\theta^*$ belonging to the sup-norm support of $\Pi$ means that
		
		\begin{align*}
			\Pi \left(\theta : \sup_{\boldsymbol {x} \in \mathcal{X}} |\theta(\boldsymbol {x}) - \theta^*(\boldsymbol {x})| < \varepsilon \right) > 0, \text{ for any } \varepsilon > 0,
		\end{align*}
		
		\noindent where $\Pi$ is the prior density for $\theta$. Thus, the prior needs to provide positive density for all neighborhoods around the parameter value that minimizes the KL divergence from the true model. In this paper, multivariate normal and $\mathcal T$ priors are assumed for the parameters, and since they have support all over $\mathbb{R}^p$, the assumption is satisfied.
		
		Concerning Assumption \ref{ass:Rama4}, \citet{Ramamoorthi2015} gives a sufficient condition consisting of the following two clauses,
		
		\begin{enumerate}[label = (\alph*)]
			\item $f_{t_1} / f_{t_2}$ is continuous in $(t_1, t_2)$ for each $y$, and 
			\item the true density $f_{0 \boldsymbol {x}}$ is continuous in $\boldsymbol {x}$ for each $y$, and can be bounded by an integrable function in $y$. 
		\end{enumerate}
		
		\noindent Regarding the latter, $Y$ is a discrete random variable, and so $f_{0\boldsymbol {x}}$ is bounded on $[0, 1]$ for each possible $y$. Continuity in $\boldsymbol {x}$ has to be assumed. As for (a), the likelihood ratio is given by
		
		\begin{align*}
			\frac{f_{t_1}}{f_{t_2}}(y) 
			& = \left[\frac{\operatorname{expit}(t_1)}{\operatorname{expit}(t_2)}\right]^y
			\left[\frac{1 - \operatorname{expit}(t_1)}{1 - \operatorname{expit}(t_2)}\right]^{1-y} \\
			& = \left[\frac{1 + e^{-t_2}}{1 + e^{-t_1}}\right]^y
			\left[\frac{e^{-t_1}}{1 + e^{-t_1}}\right]^{1-y} 
			\left[\frac{e^{-t_2}}{1 + e^{-t_2}}\right]^{-(1-y)},
		\end{align*}
		
		\noindent and is continuous in $(t_1, t_2)$, seeing as it is a composition of continuous functions.
		
		For the last assumption, $\log f_{\theta_{\boldsymbol {x}}} / f_{\theta^*_{\boldsymbol {x}}}$ being uniformly bounded is a sufficient condition according to \citet{Ramamoorthi2015}. Taking the logarithm of the likelihood ratio given above yields
		
		\begin{align*}
			\log \left[ \frac{f_{\theta_{\boldsymbol {x}}}}{f_{\theta^*_{\boldsymbol {x}}}}(y) \right] 
			& = \log \left\{\left[\frac{1 + e^{-\theta^*_{\boldsymbol {x}}}}{1 + e^{-\theta_{\boldsymbol {x}}}}\right]^y
			\left[\frac{e^{-\theta_{\boldsymbol {x}}}}{1 + e^{-\theta_{\boldsymbol {x}}}}\right]^{1-y} 
			\left[\frac{e^{-\theta^*_{\boldsymbol {x}}}}{1 + e^{-\theta^*_{\boldsymbol {x}}}}\right]^{-(1-y)}\right\}\\
			& = y \log \left[\frac{1 + e^{-\theta^*_{\boldsymbol {x}}}}{1 + e^{-\theta_{\boldsymbol {x}}}}\right] 
			+ (1 - y) \log \left[\frac{1 + e^{-\theta^*_{\boldsymbol {x}}}}{1 + e^{-\theta_{\boldsymbol {x}}}} \right] +
			(1 - y)(\theta^*_{\boldsymbol {x}} - \theta_{\boldsymbol {x}}) \\
			& = (1-y) (\theta^*_{\boldsymbol {x}} - \theta_{\boldsymbol {x}}) +  \log \left[\frac{1 + e^{-\theta^*_{\boldsymbol {x}}}}{1 + e^{-\theta_{\boldsymbol {x}}}}\right].
		\end{align*}
		\noindent It is given by assumption that each element of $\Theta$ maps to $[-K, K]$, meaning that $1 < 1 + \exp \left\{- \theta_{\boldsymbol {x}} \right\} < 1  + \exp \{K\}$ for every $\boldsymbol {x} \in \mathcal{X}$. Thus, since the logarithm is an increasing function,
		
		\begin{align*}
			\log \left[\frac{1 + e^{-\theta^*_{\boldsymbol {x}}}}{1 + e^{-\theta_{\boldsymbol {x}}}}\right] 
			\leq \log \left[\frac{1 + e^{-\theta^*_{\boldsymbol {x}}}}{1}\right] 
			\leq\log \left[1 + e^{K}\right] < \infty.
		\end{align*}
		
		\noindent Moreover, $|\theta^*_{\boldsymbol {x}} - \theta_{\boldsymbol {x}}| \leq 2K$ for every $\boldsymbol {x} \in \mathcal{X}$, which gives the required uniform boundedness.

		\subsection*{Linear regression with $\mathcal T$ prior} 
		As in the previous section, the $\mathcal T$ prior has positive density all over $\mathbb{R}^p$, meaning that it satisfies the relevant conditions on the prior density. Thus, what remains is to check whether the assumptions admit the normal linear model.

		Assumptions \ref{ass:Rama1}--\ref{ass:Rama2} are satisfied by the linear model, and verification follows from the same argument as above. Similarly, the family  $\mathcal{F}$ is defined as
		
		\begin{align*}
			\mathcal{F} = \left\{ \frac{1}{\sqrt{2\pi \sigma^2}} e^{-(y - t)^2/2\sigma^2} : t \in [-K, K] \right\},
		\end{align*}
		
		\noindent where $t$ is bounded due to the fact that it is a linear combination of finitely many terms.

		Assumption \ref{ass:Rama4} can be verified using the same sufficient conditions as previously. Starting with the likelihood ratio, it is given by 
		
		\begin{align*}
			\frac{f_{t_1}}{f_{t_2}} = \frac{\sigma_2}{\sigma_1} \exp \left\{-\frac{1}{2} \left[\frac{(y - t_1)^2}{\sigma^2_1} - \frac{(y - t_2)^2}{\sigma^2_2} \right] \right\}.
		\end{align*}
		
		\noindent Thus, it is continuous in $(t_1, t_2)$ for every $y$ by virtue of being a composition of continuous functions. Regarding the true density $f_{0\boldsymbol {x}}$, continuity in $\boldsymbol  x$ has to be assumed once more. Moreover, the normal density is bounded by $e^{-(x - \mu)^2/2\sigma^2}/\sigma$, which is integrable.

		For assumption \ref{ass:Rama5}, it is again sufficient that $\log f_{\theta_{\boldsymbol x}} /  f_{\theta^*_{\boldsymbol x}}$ is uniformly bounded. Substituting $\theta_{\boldsymbol x}$ and $\theta^*_{\boldsymbol x}$ into the previous expression and taking the log gives
		
		\begin{align*}
			\log \left(\frac{f_{\theta_{\boldsymbol x}}}{f_{\theta^*_{\boldsymbol x}}} \right) 
			= \log \left(\frac{\sigma_2}{\sigma_1}\right) -\frac{1}{2} \left[\frac{(y - \theta_{\boldsymbol x})^2}{\sigma^2_1} - \frac{(y - \theta^*_{\boldsymbol x})^2}{\sigma^2_2} \right]. 
		\end{align*}
		
		\noindent Again, we assume the model is finite-dimensional, so both $\theta_{\boldsymbol x}$ and $\theta^*_{\boldsymbol x}$ have to be bounded, say by some constant $K$. Thus, 
		
		\begin{align*}
			\frac{(y - \theta_{\boldsymbol x})^2}{\sigma^2_1} \leq \frac{\max\left\{(y - K)^2 , (y + K)^2 \right\}}{\sigma^2_1},
		\end{align*}
		
		\noindent and the same goes for the other term. Thus, combining with the fact that the squares are always positive, it follows that
		
		\begin{align*}
			-\frac{\max\left\{(y - K)^2 , (y + K)^2 \right\}}{\sigma^2_2} \leq \left[\frac{(y - \theta_{\boldsymbol x})^2}{\sigma^2_1} - \frac{(y - \theta^*_{\boldsymbol x})^2}{\sigma^2_2} \right] \leq \frac{\max\left\{(y - K)^2 , (y + K)^2 \right\}}{\sigma^2_1},
		\end{align*}
		
		\noindent and thus that
		
		\begin{align*}
			C - \frac{ \max\left\{(y - K)^2 , (y + K)^2 \right\} }{2\sigma^2_1} \leq \log \left(\frac{f_{\theta_{\boldsymbol x}}}{f_{\theta^*_{\boldsymbol x}}} \right) \leq C + \frac{ \max\left\{(y - K)^2 , (y + K)^2 \right\} }{2\sigma^2_2},
		\end{align*}
		
		\noindent where $C = \log(\sigma_2 / \sigma_1)$ is assumed constant.
	\end{appendices}
	
\end{document}